\documentclass[oneside,reqno,12pt,a4paper]{amsart}
\usepackage[T1]{fontenc}
\usepackage[australian]{babel}
\usepackage[margin=2.5cm]{geometry}
\usepackage{amsmath,amsfonts,amssymb,amscd,amsthm}
\usepackage[usenames,dvipsnames,svgnames]{xcolor}
\usepackage[utf8]{inputenc} 
\usepackage{graphicx} 
\graphicspath{ {} }
\usepackage{caption}
\usepackage{subcaption}
\usepackage{doi}
\usepackage{bm}
\usepackage{amssymb}
\usepackage{upgreek}
\usepackage{mathrsfs}
\usepackage{makecell}
\usepackage{hyperref}
\usepackage{acronym}
\usepackage{listings}
\usepackage{textgreek}
\usepackage{url}
\usepackage{color}
\usepackage{framed}
\usepackage{verbatim}
\usepackage{fancyvrb}
\usepackage{newtxtext}
\usepackage{newtxmath} 
\usepackage{microtype} 
\usepackage[final]{pdfpages}
\usepackage{tikz}
\usepackage{booktabs}
\usepackage{multirow}
\usepackage{mathtools}
\usepackage[normalem]{ulem}
\mathtoolsset{showonlyrefs}
\usepackage{csquotes}
\usepackage[backend=biber,
defernumbers=true,
maxbibnames=5,
style=numeric-comp,
isbn=false,
bibencoding=utf8,
safeinputenc, 
url=true,
doi=true,
giveninits=true]{biblatex}
\hypersetup{ 
breaklinks=true,
bookmarksopen=true,
pdftitle={Unfitted Darcy},    
pdfauthor={Badia et al},     
colorlinks=true,       
linkcolor=black,          
citecolor=blue,        
filecolor=black,      
urlcolor=blue           
}
\definecolor{bg}{rgb}{0.93,0.93,0.93}

\AtEveryBibitem{\clearfield{month}}
\AtEveryBibitem{\clearfield{day}}
\AtEveryBibitem{%
	\ifboolexpr{ not test {\iffieldundef{doi}} }{%
		\clearfield{url}%
		\clearfield{urldate}%
	}{}
}

\newtheorem{theorem}{Theorem}[section]
\newtheorem{lemma}[theorem]{Lemma}

\newtheorem{assumption}[theorem]{Assumption}
\newtheorem{remark}[theorem]{Remark}

\numberwithin{equation}{section}
\numberwithin{figure}{section}
\numberwithin{table}{section}

\usepackage[export]{adjustbox}
\acrodef{pde}[PDE]{partial differential equation}
\acrodef{fe}[FE]{finite element}
\acrodef{fem}[FEM]{finite element method}
\acrodef{dof}[DOF]{degree of freedom}
\acrodefplural{dof}[DOFs]{degrees of freedom}
\acrodef{agfe}[AgFE]{aggregated finite element}
\acrodef{agfem}[AgFEM]{aggregated finite element method}
\acrodef{rt}[RT]{Raviart--Thomas}
\acrodef{bdm}[BDM]{Brezzi--Douglas--Marini}

\newcommand{\tnor}[1]{{\left\vert\kern-0.25ex\left\vert\kern-0.25ex\left\vert #1 
\right\vert\kern-0.25ex\right\vert\kern-0.25ex\right\vert}}
 \newcommand{\jump}[1]{[#1]}

\definecolor{shadecolor}{gray}{.92}
\definecolor{incolor}{rgb}{0,0,.7} 
\definecolor{outcolor}{rgb}{.65,0,0}
\definecolor{syntaxcolor}{rgb}{.65,0,0}

\def\jumpl{\lbrack\kern-0.1em\!\lbrack}
\def\jumpr{\rbrack\kern-0.1em\!\rbrack}
\def\jump#1{\jumpl #1 \jumpr}

\newcommand{\dv}[1]{\mathrm{div} \, #1}
\def\dom{\Omega}
\def\domact{\Omega_h}
\def\domext{\Omega^{\mathrm{ext}}}

\def\domin{\Omega_h^{\mathrm{in}}}

\def\bouu{\Gamma_{u}}
\def\boup{\Gamma_p} 
\def\Fstab{\mathcal{F}_{h}^{\mathrm{stab}}}

\def\Vh{V_{\mathrm{d},h}}

\newcommand{\Vhc}[2]{V_{#1,h}^{#2}}

\def\Qh{V_{\mathrm{0},h}}

\def\V{\mathcal{V}_{\mathrm{d}}}

\def\Q{\mathcal{V}_{\mathrm{0}}}

\def\Hstr{\mathscr{H}^{\mathrm{*}}_h}
\def\Hv{\mathscr{H}^{\mathrm{d}}_h}
\def\HvT{\mathscr{H}^{\mathrm{d}}_{h,T}}
\def\Hq{\mathscr{H}^{0}_h}

\def\extv{\mathscr{E}^{\mathrm{d}}_{h}}
\def\extq{\mathscr{E}^{0}_{h}}
\def\Pv{\mathscr{P}^{\mathrm{d}}_h}
\def\Pq{\mathscr{P}^{0}_{h}}

\def\Iv{\Pi_{\mathrm{d},h}}
\def\IvT{\Pi_{\mathrm{d},h,T}}
\def\Iq{\Pi_{0,h}}

\def\xivh{\tau_{\mathrm{d}}}
\def\xiqh{\tau_{0}}

\def\ahf{a_h}
\def\shpf{s^0_{h}}
\def\shuf{s^{\mathrm{d}}_h}
\def\bbf{b}

\def\piulocal{\pi_{\mathrm{d},h,T}}

\def\piqst{\pi_{h}^{0}}

\def\uh{u_h}

\newcommand{\Bh}{{b}_h}  
\newcommand{\dimnor}[1]{#1}

\newcommand{\change}[2]{%
  \begingroup
  \if\relax\detokenize{#1}\relax\else\color{red}{#1}\fi%
  \if\relax\detokenize{#2}\relax\else{\textcolor{red}{\sout{#2}}}\fi%
  \endgroup
}

\def\Th{\mathcal{T}_h}  
\def\Thagb{\mathcal{T}_h^{\partial,\mathrm{ag}}}  
\def\Thag{\mathcal{T}_h^{\mathrm{ag}}}  

\begin{document}
\title[Unfitted finite element methods for Darcy]{Divergence-free unfitted finite element discretisations \\ for the Darcy problem}
\author{Santiago Badia$^{1,*}$}
\thanks{$^1$ School of Mathematics, Monash University, Clayton, Victoria 3800, Australia. $^2$ School of Computing, The Australian National University, Canberra ACT 2600, Australia. $^3$ Department of Mathematics, KTH Royal Institute of Technology, SE-100 44 Stockholm, Sweden. $^4$ Universidad Adventista de Chile, Casilla 7-D, Chill\'an, Chile. $^*$ Corresponding author. Email: santiago.badia@monash.edu.}
\author{Anne Boschman$^{1}$}
\author{Alberto F. Mart\'{\i}n$^2$}
\author{Erik Nilsson$^3$}
\author{Ricardo Ruiz-Baier$^{1,4}$}
\author{Sara Zahedi$^3$}
\date{}
\begin{abstract}
We develop an unfitted compatible finite element discretisation for the Darcy problem based on $H(\mathrm{div})$-conforming flux spaces and discontinuous pressure spaces. The method is designed to preserve pointwise discrete mass conservation while remaining robust in the presence of arbitrarily small cut cells arising from unfitted meshes. Robustness is achieved by combining an $L^2$-stabilisation of the flux with an additional mixed-term stabilisation that enhances pressure control without destroying the local conservation structure. We consider both cell-wise (bulk) and face-based ghost-penalty realisations of the stabilisation. Mixed boundary conditions are handled by weak imposition of both flux and pressure traces on unfitted boundaries. We prove stability and a priori error estimates with constants independent of the cut configuration, and establish pressure-robust flux error bounds in the case of pure pressure boundary conditions. We also introduce an augmented Lagrangian variant that improves control of the conservation constraint and is amenable to efficient preconditioning strategies. Numerical experiments for a range of cut configurations, boundary-condition regimes and parameter choices confirm the theoretical results, demonstrating optimal convergence, cut-independent conditioning and mass conservation up to solver tolerance.
\end{abstract}
 
\maketitle 

\section{Introduction}

Standard \ac{fe} methods typically require body-fitted triangulations of the physical domain in which one wants to discretise the \ac{pde}. These meshes are difficult to generate for complex or moving domains. 
In contrast, unfitted (or immersed) methods decouple the discretisation of the geometry from that of the \ac{pde} and employ simple background meshes that do not need to conform to the domain boundary \cite{acta25}. Instead of mesh generation algorithms, the geometrical discretisation relies on algorithms that compute the intersection between the background mesh and the geometry \cite{Badia-JComputPhys-2022k,Martorell-JComputPhys-2024h}, which is an embarrassingly parallel cell-wise operation. On these intersections, one must design quadrature rules using, for example, moment-fitting techniques \cite{saye2017implicit}. While this approach is highly efficient, it introduces the small cut-cell problem. This occurs when arbitrarily small intersections between the mesh and the boundary lead to severe numerical instabilities and ill-conditioning \cite{dePrenter-ArchComputMethodsEng-2023j}. To restore robustness, several stabilisation strategies have been developed. Ghost-penalty schemes stabilise the system by penalising derivative jumps across facets cutting the boundary \cite{burman_cutfem_2015}. Alternatively, cell-aggregation techniques, such as aggregated finite element methods (AgFEM), address the problem by extending well-posed degrees of freedom to ill-posed ones through discrete extension operators \cite{Badia2018-vw}. Most work on unfitted methods has focused on grad-conforming approximations and nodal Lagrangian FEM.

$H(\mathrm{div})$-conforming discretisations of the mixed Poisson or Darcy problem have been considered in \cite{Puppi-arXivmathNA-2021o}, where a ghost-penalty pressure-stabilisation term is added to achieve robustness with respect to small cut cells. However, the resulting method does not satisfy pointwise discrete mass conservation. Exact mass conservation has been attained in \cite{Frachon2024-fg} for an unfitted interface Darcy problem by using indefinite stabilisation terms that do not perturb the discrete mass-conservation equation, while the boundary remains body-fitted. Using ideas similar to those in \cite{Burman-SIAMJNumerAnal-2024q}, exact mass conservation is attained in \cite{Lehrenfeld-MathComput-2024j} in the boundary-unfitted setting. The method proposed in \cite{Lehrenfeld-MathComput-2024j} relies on an extension of the mass-conservation equation to an active mesh instead of ghost-penalty stabilisation. Although the method is analysed for pressure boundary conditions, an extension to mixed boundary conditions using Lagrange multipliers is also discussed. In addition, the weak imposition of flux boundary conditions in $H(\mathrm{div})$-conforming discretisations has been considered in \cite{Burman-JNumerMath-2022z} for the Darcy problem via a consistent penalty method; this method was previously applied to an interface problem in \cite{DAngelo-EsaimMathModelNumerAnal-2012r}.

In this work, we propose a novel unfitted finite element method for Darcy flow with mixed boundary conditions in the boundary-unfitted setting that satisfies pointwise discrete mass conservation and is robust with respect to cut-cell locations. The method incorporates weak imposition of flux and pressure boundary conditions at unfitted boundaries and relies on an $L^2$-stabilisation of the flux together with a mixed stabilisation term that provides enhanced control of the pressure, using ideas similar to those in \cite{Frachon2024-fg} for interface problems. The proposed stabilisation is a ghost-penalty method, which can be defined in terms of face jumps (as in \cite{Frachon2024-fg}) or bulk penalty terms (as in, for example, the original stabilisation proposed in \cite{Burman2010-tf} or the modification proposed in \cite{Badia2022-zq}). We prove that the method is stable, convergent and satisfies pointwise discrete mass conservation in the boundary-unfitted case with mixed boundary conditions. 
In addition, for the case of pure pressure boundary conditions, we prove a novel pressure-robust error estimate for the flux. We also consider an augmented Lagrangian formulation of the proposed method, which may be useful for designing effective preconditioners and for improving mass conservation in the presence of mixed boundary conditions.

{
To the best of our knowledge, this is the first rigorous stability and error analysis for a div-conforming Darcy discretisation with flux boundary conditions in the boundary-unfitted setting. The closest prior work, \cite{Lehrenfeld-MathComput-2024j}, provides analysis only for pure pressure boundary conditions; although its stabilisation differs from ours, the two approaches are related~\cite[Lemma~1]{Lehrenfeld-MathComput-2024j}.
}
{
The extension to flux boundary conditions is achieved through a deliberately non-symmetric Nitsche formulation designed to preserve pointwise mass conservation. The theoretical results are also new in the body-fitted setting; a body-fitted version of the Nitsche method was proposed in \cite{Burman-JNumerMath-2022z}, but without numerical analysis.
The weak imposition of flux boundary conditions in the unfitted case substantially complicates the analysis: one cannot define a test space that both vanishes on the flux boundary and satisfies a pressure inf-sup condition. Instead, we establish a global inf-sup condition.
}

\medskip\noindent\textbf{Outline of the paper.}
We begin with the geometrical discretisation in Section~\ref{sec:geom}. Next, we introduce the method and discuss its main features in Section~\ref{sec:formulation}. The stability analysis and a priori error estimates are presented in Sections~\ref{sec:stability} and~\ref{sec:error-estimates}, respectively. An augmented Lagrangian formulation of the method is presented in Section~\ref{sec:aug-lag}. The cell-wise (bulk) and face-based ghost-penalty stabilisation mechanisms are described in Section~\ref{sec:stab}. Finally, numerical experiments illustrating convergence and robustness with respect to cut configurations are reported in Section~\ref{sec:num-results}. We draw some conclusions in Section~\ref{sec:conclusions}.

\section{Geometrical discretisation}\label{sec:geom}

Let $\Omega \subset \mathbb{R}^d$, with $d$ the space dimension, be homeomorphic to an open polyhedral domain. $\Omega$ represents the physical domain of the \ac{pde} to be solved. Standard \ac{fe} methods are formulated on a so-called \emph{body-fitted} mesh, which is a partition of $\Omega$ (or of an approximation of it). Unfitted discretisation techniques, by contrast, do not rely on a \emph{body-fitted} mesh. Instead, the physical domain is embedded into a background domain $\Omega^{\mathrm{bg}}$, such that $\Omega \subset \Omega^{\mathrm{bg}}$.

Let $\mathcal{T}_{h}^{\mathrm{bg}}$ be a conforming, quasi-uniform and shape-regular partition of $\Omega_{h}^{\mathrm{bg}}$. We denote by $h_T$ the diameter of a cell $T \in \mathcal{T}_{h}^{\mathrm{bg}}$ and the characteristic mesh size is $h \doteq \max_{T \in \mathcal{T}_{h}^{\mathrm{bg}}} h_T$. Next, we introduce some geometrical definitions in order to define unfitted \ac{fe} discretisations.

{First, we let $\{\mathcal{T}_{h}, \mathcal{T}_{h}^{\mathrm{out}}\}$ denote a partition of $\mathcal{T}_{h}^{\mathrm{bg}}$ into cells intersecting $\Omega$ (active cells) and cells not intersecting $\Omega$ (outer cells), respectively. Since outer cells do not play any role in the discretisation, they are discarded. We further partition $\mathcal{T}_{h}$ into $\{\mathcal{T}_{h}^{\mathrm{in}}, \mathcal{T}_{h}^{\mathrm{cut}}\}$, where}
\begin{equation}
  {\mathcal{T}_{h}^{\mathrm{in}} \doteq \left\{ T \in \mathcal{T}_{h} : T \subset \Omega \right\}, \qquad \mathcal{T}_{h}^{\mathrm{cut}} \doteq \mathcal{T}_{h} \setminus \mathcal{T}_{h}^{\mathrm{in}}.}
\end{equation}
{The cells in $\mathcal{T}_{h}^{\mathrm{in}}$ are called \emph{interior} cells, while the cells in $\mathcal{T}_{h}^{\mathrm{cut}}$ are called \emph{cut} cells.}

We define the active domain $\Omega_h$ as the interior of the closure of $\bigcup_{T \in \mathcal{T}_{h}} T$ and denote by $\Omega_h^{\mathrm{in}}$ and $\Omega_h^{\mathrm{cut}}$ the unions of cells in $\mathcal{T}_{h}^{\mathrm{in}}$ and $\mathcal{T}_{h}^{\mathrm{cut}}$, respectively.

To complete the geometrical definitions, we extend the previous classifications to the \emph{facets}. Let $\mathcal{F}_{h}$ denote the set of all facets in the active mesh $\mathcal{T}_{h}$, where a \emph{facet} refers to 1-faces (edges) in 2D and 2-faces (faces) in 3D. We also define $\mathcal{F}_h^{\mathrm{cut}}$ as the set of interior facets of the active mesh that are faces of at least one cut cell in $\mathcal{T}_h^{\mathrm{cut}}$.

Based on this classification, we introduce an aggregated mesh $\mathcal{T}_{h}^{\mathrm{ag}}$, obtained by associating each cell in $\mathcal{T}_{h}^{\mathrm{cut}}$ with exactly one cell in $\mathcal{T}_{h}^{\mathrm{in}}$, called the \emph{root} cell, through a chain of facet-adjacent cells; see \cite{Badia2018-vw,LarZah23}. Each cell in $\mathcal{T}_{h}$ belongs to one, and only one, aggregate, while cells in $\mathcal{T}_{h}^{\mathrm{in}}$ that do not absorb cut cells define singleton aggregates. We denote by $\mathcal{T}_{h}^{\partial,\mathrm{ag}}$ the set of non-trivial boundary aggregates, i.e., the aggregates containing at least one cell in $\mathcal{T}_{h}^{\mathrm{cut}}$. The goal of aggregation is to obtain patches that intersect $\Omega$ with measure bounded away from zero. We assume that every aggregate has diameter uniformly bounded by $h$. This aggregate structure will be used in the proposed stabilisation methods below.

{For the numerical experiments in Section~\ref{sec:num-results}, we use a slightly more general definition of interior and cut cells, depending on a parameter $\delta \in (0,1]$:}
\begin{equation}
  {\mathcal{T}_{h}^{\mathrm{in}} \doteq \left\{ T \in \mathcal{T}_{h} : \frac{|T \cap \Omega|}{|T|} \geq \delta \right\}, \qquad \mathcal{T}_{h}^{\mathrm{cut}} \doteq \mathcal{T}_{h} \setminus \mathcal{T}_{h}^{\mathrm{in}}.}
\end{equation}
{The analysis presented in this work can readily be extended to this $\delta$-based classification.}

{Finally, we introduce a fixed smooth domain $\domext$ such that $\domact \subset \domext$ for all $0 < h \leq h_0$. This outer domain will be used in the stability and error analysis below.}

\section{Problem formulation and preliminaries}\label{sec:formulation}

As a model problem, we consider Darcy's system in $\Omega$. We consider a partition of the boundary $\Gamma = \partial \Omega$ into $\bouu$ and $\boup$, on which the flux and pressure traces are prescribed, respectively. The outward unit normal vector to the boundary is denoted by $n$. The problem reads as follows: find the flux $u : \Omega \to \mathbb{R}^d$ and the pressure $p : \Omega \to \mathbb{R}$ such that
\begin{align}\label{eq:darcy}
  {\eta} u + \nabla p = f \quad  \text{in} \ \Omega, \quad \dv{u} = -g \quad  \text{in} \ \Omega, \quad 
  u \cdot n = u_\Gamma \quad  \text{on} \ \Gamma_{u}, \quad 
  p = p_\Gamma \quad  \text{on} \ \Gamma_p,
\end{align}
where ${\eta}$ is the inverse permeability tensor, $f$ is the source term, $g$ is the divergence source term and $u_\Gamma$ and $p_\Gamma$ are the boundary traces for the normal flux and pressure, respectively.
When $\Gamma_p = \emptyset$, i.e., when pure flux boundary conditions are considered, the solvability constraint $\int_{\partial \Omega} {u}_\Gamma = {-}\int_\Omega g$ must be satisfied and the pressure is only defined up to a constant.

Let $H_{u_\Gamma}(\mathrm{div},\Omega)$ and $H_0(\mathrm{div},\Omega)$ be the spaces of functions in $H(\mathrm{div},\Omega)$ whose normal trace on $\Gamma_u$ is equal to $u_\Gamma$ and zero, respectively. Assume that $\eta \in L^\infty(\Omega)$ and $\eta \geq \eta_0 > 0$ a.e. in $\Omega$ for some constant $\eta_0 > 0$, $f \in [L^2(\Omega)]^d$, $g \in L^2(\Omega)$, $p_\Gamma \in H^{1/2}(\Gamma_p)$ and $u_\Gamma \in H^{-1/2}(\Gamma_u)$. The weak form of \eqref{eq:darcy} reads: find $u \in H_{u_\Gamma}(\mathrm{div},\Omega)$ and $p \in L^2(\Omega)$ such that
\begingroup
\mathtoolsset{showonlyrefs=false}
\begin{subequations}\label{eq:weak-form}
\begin{align}
  a(u, v) + b(v, p) &= ( f, v )_\Omega - \langle v \cdot n, p_\Gamma \rangle_{\Gamma_p}, & \forall v \in H_0(\mathrm{div},\Omega), \\
  b(u, q) &= ( g, q )_\Omega, & \forall q \in L^2(\Omega),
\end{align}
\end{subequations}
\endgroup
where 
\begin{align}
  a(u, v) \doteq (\eta u, v)_\Omega, \quad \text{ and } \quad b(v, p) \doteq - (\dv{v}, p)_\Omega. 
\end{align}
Here, $(\cdot,\cdot)_\omega$ denotes the $L^2$ inner product over $\omega$, and $\langle \cdot, \cdot \rangle_{\omega}$ denotes the duality pairing between $H^{-1/2}(\omega)$ and $H^{1/2}(\omega)$. When $\Gamma_p = \emptyset$, we seek $p \in L^2_0(\Omega)$, i.e., the space of functions in $L^2(\Omega)$ with zero mean in $\Omega$, test the second equation in \eqref{eq:weak-form} with $q \in L^2_0(\Omega)$, and assume that the compatibility condition $\int_{\partial \Omega} u_\Gamma = -\int_\Omega g$ on the data is satisfied.

\subsection{Discrete spaces}

Let $\Vh \subset H(\mathrm{div}, \Omega_h)$ and $\Qh \subset L^2(\Omega_h)$ be $H(\mathrm{div})$-conforming and $L^2$-conforming \ac{fe} spaces defined on $\mathcal{T}_{h}$, respectively. Let $\mathbb{P}_k$ denote the space of polynomials of degree at most $k$ and let $\mathbb{Q}_k$ denote the space of tensor-product polynomials of degree at most $k$ in each variable. We make the following assumption on these discrete spaces.
\begin{assumption}\label{as:discrete-spaces}
The global flux space $\Vh$ can be defined in terms of a space of vector fields $\V$ with polynomial components as follows:
\begin{equation}\label{ref:Vh}
  \Vh = \left\{ v_h \in H(\mathrm{div},\Omega_h) : v_h|_T \in \V(T), \forall T \in \mathcal{T}_h \right\},
\end{equation}
where $\V(T)$ denotes the restriction of $\V$ to $T \in \mathcal{T}_h$.
The pressure space $\Qh$ consists of discontinuous piecewise polynomials $\Q(T)$ on each $T \in \mathcal{T}_h$. These two spaces determine a discrete de Rham complex between $d-1$ and $d$-forms with bounded cochain interpolators $\Iv$ and $\Iq$~\cite{Arnold2019-hi}. 

Let $k_u, k_p \in \{0,1,\ldots\}$ be the largest values such that $[\mathbb{P}_{k_u}(T)]^d \subset \V(T)$ and $\mathbb{P}_{k_p}(T) \subset \Q(T)$ for all $T \in \mathcal{T}_h$, respectively.  We refer to $k_u$ and $k_p$ as the polynomial degrees of $\Vh$ and $\Qh$, respectively. We assume that the interpolators $\Iv$ and $\Iq$ satisfy the following approximation properties: given $u \in [H^{r}(\Omega_h)]^d$ and $p \in H^{t}(\Omega_h)$, it holds
\begin{align}
&\| u - \Iv(u) \|_{\dimnor{L^2(\Omega_h)}} \lesssim h^{r} \| u \|_{\dimnor{H^{r}(\Omega_h)}}, \quad \hbox{for } r = 0, \ldots, k_u+1,  \\ 
&\| p - \Iq(p) \|_{L^2(\Omega_h)} \lesssim h^{t} \| p \|_{H^{t}(\Omega_h)}, \quad \hbox{for } t = 0, \ldots, k_p+1.
\end{align}
\end{assumption}

This choice implies a discrete inf-sup condition, i.e., the restriction of $\mathrm{div}$ to $\Vh$ is surjective onto $\Qh$, which is necessary for the stability of the mixed formulation. In particular, for such a spatial discretisation, it holds that $\dv{\Vh} = \Qh$, due to the trivial exactness of the relevant portion of the discrete complex. Possible choices for these spaces can be found in~\cite{Arnold2019-hi} for simplicial and quadrilateral meshes. They include Raviart--Thomas and Brezzi--Douglas--Marini spaces, as well as the serendipity \acp{fe} in~\cite{Arnold2014-df}. The proposed methodology is agnostic to the mesh topology; in Section~\ref{sec:num-results}, we use both quadrilateral and triangular meshes. In particular, in the numerical experiments, we use $\mathbb{RT}_k \times \mathbb{Q}_{k}$ on quadrilateral meshes and $\mathbb{RT}_k \times \mathbb{P}_{k}$ and $\mathbb{BDM}_{k+1} \times \mathbb{P}_{k}$ on triangular meshes, for $k \geq 0$. Note that, since $\dv{\Vh} = \Qh$, $k_p \leq k_u$ and, given $\dv{u} \in H^t(\Omega_h)$, it holds that:
\begin{equation}
 \| \dv{u} - \Iq(\dv{u}) \|_{L^2(\Omega_h)} \lesssim h^{t} \| \dv{u} \|_{H^{t}(\Omega_h)}, \quad \hbox{for } t =  0, \ldots, k_p+1.
\end{equation} 

\subsection{Discrete formulation}

Next, we consider the discrete formulation of \eqref{eq:weak-form}. First, we weakly impose the flux boundary conditions since the boundary $\Gamma$ is embedded in $\Omega_h$. Second, we add stabilisation terms to the discrete problem to ensure stability and robustness of the solution with respect to small cut cells. In particular, we make the following assumption on the stabilisation terms. Later, we propose several stabilisation terms that satisfy this assumption. We use $a \lesssim b$ (respectively, $a \gtrsim b$) to denote that $a \leq C b$ (respectively $a \geq C b$) for some constant $C > 0$ independent of the characteristic mesh size $h$ and the cut location; we use $a \simeq b$ if both $a \lesssim b$ and $b \lesssim a$ hold. For simplicity of exposition, we take $\eta$ to be scalar and equal to one.

\begin{assumption}\label{as:stabilisation-terms}
  Let $\shuf$ and $\shpf$ be the stabilisation terms for the flux and pressure, respectively. The stabilisation terms satisfy the following properties:
  \begin{subequations}
  \begin{align}
    \| w_h \|^2_{\dimnor{L^2(\Omega)}} + \shuf(w_h,w_h) &\simeq \| w_h \|_{\dimnor{L^2(\Omega_h)}}^2,  & \forall w_h \in \Vh, \label{eq:stab-assump-1} \\
    \| r_h \|^2_{L^2(\Omega)} + \shpf(r_h,r_h) &\simeq \| r_h \|_{L^2(\Omega_h)}^2,  & \forall r_h \in \Qh. \label{eq:stab-assump-2}
  \end{align}
  \end{subequations}
  Furthermore, the stabilisation terms satisfy the following weak consistency estimates: 
  \begin{subequations}\label{eq:weak-consistency}
  \begin{align}
    \shuf(\Iv(w),v_h) &\lesssim h^{s}\| w \|_{\dimnor{H^{s}(\domext)}} \| v_h \|_{\dimnor{L^2(\Omega_h)}}, & 0 \leq s \leq k_u+1, \label{eq:weak-consistency-1} \\
    \shpf(\Iq(r),q_h) &\lesssim h^{t} \| r \|_{H^{t}(\domext)} \| q_h \|_{L^2(\Omega_h)}, & 0 \leq t \leq k_p+1, \label{eq:weak-consistency-2}
  \end{align}
  \end{subequations}
  for any $w \in [H^{s}(\domext)]^d$, $r \in H^{t}(\domext)$, $v_h \in \Vh$ and $q_h \in \Qh$. 
  In addition, the pressure stabilisation vanishes on constants, i.e., $\shpf(c,q_h)=0$, for every constant $c$ and every $q_h \in \Qh$.
\end{assumption}
{Note that, taking $s=0$ in \eqref{eq:weak-consistency-1} and $t=0$ in \eqref{eq:weak-consistency-2}, and using that the interpolators act as the identity on the discrete spaces, the weak consistency implies continuity of the stabilisation terms in the $L^2(\Omega_h)$-norm in the discrete spaces.}

In order to state the weak form of \eqref{eq:darcy} with weak imposition of flux boundary conditions, we define the discrete forms:
\begin{align}
   \ahf(u_h,v_h) & \doteq a(u_h,v_h) + \gamma h^{-1} (u_h \cdot n, v_h \cdot n)_{\bouu} + \xivh \shuf(u_h,v_h), \\
   \bbf_h(v_h,p_h) & \doteq b(v_h,p_h) - \xiqh \shpf(\dv{v_h},p_h), \label{eq:b}\\
  \tilde\bbf_h(v_h,p_h) &\doteq \tilde b(v_h,p_h) - \xiqh \shpf(\dv{v_h},p_h),
\end{align}
where $\tilde b(v_h,p_h) \doteq b(v_h,p_h) + (v_h \cdot n ,p_h)_{\Gamma_{u}}$, and $\gamma, \xivh, \xiqh > 0$ are positive numerical parameters. The discrete formulation reads: find $(u_h,p_h) \in \Vh \times \Qh$ such that
\begingroup
\mathtoolsset{showonlyrefs=false}
  \begin{subequations}\label{eq:discrete-problem}
    \begin{align}
      a_h(u_h,v_h) + \tilde b_h(v_h,p_h) & = (f,v_h)_{\Omega} + \gamma h^{-1} (u_\Gamma,v_h \cdot n )_{\bouu} - (v_h \cdot n,p_\Gamma)_{\boup}, \label{eq:discrete-problem-1} \\
      b_h(u_h,q_h) & = (g,q_h)_{\Omega}, \label{eq:discrete-problem-2}
    \end{align}
    \end{subequations}
\endgroup
for all $(v_h,q_h) \in \Vh \times \Qh$. In the case $\Gamma_p = \emptyset$, the pressure trial and test spaces are replaced by $\Qh \cap L^2_0(\Omega)$.

Let us define the following stabilised $L^2$-projection $\pi_{h}^{0}: L^2(\Omega) \rightarrow \Qh$: given $r \in L^2(\Omega)$, find $\pi_{h}^{0}(r) \in \Qh$ such that
\begin{align}\label{eq:stabilised-projection}
( \pi_{h}^{0}(r),q_h)_{\Omega} + \xiqh \shpf(\pi_{h}^{0}(r),q_h) = (r,q_h)_{\Omega}, \quad \forall q_h \in \Qh.
\end{align}
Next, using this stabilised $L^2$-projection we show that the discrete formulation satisfies the following property. 
\begin{lemma}[Discrete mass conservation]\label{lem:mass-conservation}
If $\Gamma_p \neq \emptyset$, the discrete solution $u_h$ satisfies $\dv{u_h} = -\piqst(g)$ on $\Omega_h$. {If $\Gamma_p = \emptyset$, there exists a constant $c_h \in \mathbb{R}$ such that
$\dv{u_h} = -\piqst(g) + c_h$ on $\Omega_h$. For $g=0$, the discrete flux divergence  is zero for $\Gamma_p \neq\emptyset$ and constant otherwise.}
\end{lemma}
\begin{proof}
  If $\Gamma_p \neq \emptyset$, the claim follows directly from \eqref{eq:discrete-problem-2} and the definition of the projection in \eqref{eq:stabilised-projection}. 
  {Assume now that $\Gamma_p=\emptyset$. Then \eqref{eq:discrete-problem-2} holds for all $q_h \in \Qh \cap L^2_0(\Omega)$. Combining \eqref{eq:discrete-problem-2} with \eqref{eq:stabilised-projection}, we obtain}
  \begin{equation*}
    (\dv{u_h}+\pi_h^0(g),q_h)_\Omega + \xiqh \shpf(\dv{u_h}+\pi_h^0(g),q_h) = 0, \quad \forall q_h \in \Qh \cap L^2_0(\Omega).
  \end{equation*}
  {Since $\dv{u_h}+\pi_h^0(g) \in \Qh$ by Assumption~\ref{as:discrete-spaces}, we write $\dv{u_h}+\pi_h^0(g)=z_h+c_h$ with $z_h \in \Qh \cap L^2_0(\Omega)$ and $c_h$ constant. Taking $q_h=z_h$ in the previous identity and using that $(c_h,z_h)_\Omega=0$ together with $\shpf(c_h,z_h)=0$ by Assumption~\ref{as:stabilisation-terms}, we obtain}
  \begin{equation*}
    \|z_h\|_{L^2(\Omega)}^2 + \xiqh \shpf(z_h,z_h)=0.
  \end{equation*}
  {Hence, $z_h=0$, and therefore $\dv{u_h} = -\pi_h^0(g) + c_h$ on $\Omega_h$. The results for $g=0$ relies on the fact that $\piqst(0) = 0$.}
\end{proof}
\begin{remark}
  {The formulation (\ref{eq:discrete-problem}) is deliberately non-symmetric. Adding the term $(u_h \cdot n, q_h)_{\Gamma_u}$ to the mass equation would destroy the mass conservation and the property $\dv{u_h} = -\pi_h^0(g)$ in Lemma~\ref{lem:mass-conservation}.}
\end{remark}

\begin{remark}
{In the case $\Gamma_p=\emptyset$, the pressure space in the discrete formulation is replaced by $\Qh \cap L^2_0(\Omega)$. In general one does not have $\dv{\Vh}\subset \Qh \cap L^2_0(\Omega)$, since the flux boundary condition is imposed weakly and therefore the mean value of $\dv{v_h}$ is not constrained for $v_h\in\Vh$. As a consequence, the discrete conservation equation only determines the zero-mean component of $\dv{u_h}$, while its constant mode remains undetermined. {To address this issue, one can impose the additional constraint $\int_{\partial \Omega} u_h \cdot n = - \int_\Omega g$ via a Lagrange multiplier in \eqref{eq:discrete-problem}, so that the previous lemma holds with $c_h = 0$ when $\Gamma_p = \emptyset$. We have used this technique in the numerical experiments in Section~\ref{sec:num-results}. The analysis below can readily be extended to this corrected formulation, but it is not included here for conciseness.}
}
\end{remark}
\begin{remark}
  {The constant $c_h$ converges to zero as $h \to 0$. Indeed, since $\dv{u}=-g$ in $\Omega$ and the stabilised projection preserves constants, one has}
  \begin{equation*}
    |\Omega|^{1/2}|c_h| = \|\dv{u_h}+\pi_h^0(g)\|_{L^2(\Omega)} \leq \|\dv{u}-\dv{u_h}\|_{L^2(\Omega)} + \|g-\pi_h^0(g)\|_{L^2(\Omega)}.
  \end{equation*}
  {The a priori estimate in Section~\ref{sec:error-estimates} together with Lemma~\ref{lem:stabilised-projection-alt} imply that $c_h \to 0$ as $h \to 0$.}
\end{remark}

\section{Stability analysis}\label{sec:stability}
In this section, we prove the stability of the discrete problem \eqref{eq:discrete-problem}. Let us write the system in compact form as $\tilde A_h((u_h,p_h),(v_h,q_h)) = \ell_h(v_h,q_h)$, where
\begin{align}\label{eq:discrete-problem-operator}
  \tilde A_h((u_h,p_h),(v_h,q_h))  &\doteq a_h(u_h,v_h) + \tilde b_h(v_h,p_h) + b_h(u_h,q_h),\\
  \ell_h(v_h,q_h) &\doteq (f,v_h)_{\Omega} + \gamma h^{-1} (u_\Gamma,v_h \cdot n )_{\bouu} - (v_h \cdot n,p_\Gamma)_{\boup} + (g,q_h)_{\Omega}.
\end{align}

We note that the analysis does not rely on a continuous inf-sup condition whose constant is uniform with respect to $h$ over the family of active domains $\domact$. Indeed, $\domact$ is a family of $h$-dependent domains that may fail to satisfy the assumptions required at the continuous level to obtain an inf-sup constant bounded away from zero as $h \downarrow 0$.
Instead, we consider the discrete multilinear form 
\begin{equation}
  A_h((u_h,p_h),(v_h,q_h)) \doteq a_h(u_h,v_h) + b_h(v_h,p_h) + b_h(u_h,q_h)
\end{equation}
without the boundary term on $\Gamma_{u}$ and prove a discrete weak inf-sup condition in $\domact$. We then rely on this result to analyse the full multilinear form $\tilde A_h$, which contains the boundary term $(v_h \cdot n, p_h)_{\Gamma_{u}}$.

We define the norms: 
\begin{align*}
  \| u \|^2_{\mathrm{d},h} &\doteq  \|u\|^2_{\dimnor{L^2(\domact)}} + \|\dv{u}\|^2_{L^2(\domact)} +\gamma  h^{-1} \|u \cdot n\|^2_{L^2(\bouu)},  
   \\ 
   \| p \|^2_{0,h} &\doteq \| p \|^2_{L^2(\domact)} + h \|p \|_{L^2(\bouu)}^2,
\end{align*}
and note that 
\begin{equation}\label{eq:norm-equivalence-pressure}
  \|p_h\|_{L^2(\Omega_h)} \le \|p_h\|_{0,h} \lesssim \|p_h\|_{L^2(\Omega_h)}, \qquad \forall p_h \in \Qh.
\end{equation}
The lower bound is immediate from the definition. For the converse, we use the trace inequality on cut elements
\begin{equation}\label{eq:cut-trace}
  h\|\phi\|_{L^2(\partial \Omega \cap T)}^2 \lesssim \|\phi\|_{L^2(T)}^2 + h^2\|\nabla \phi\|_{L^2(T)}^2, \qquad \forall T \in \mathcal{T}_{h}^{\mathrm{cut}},
\end{equation}
which holds for every $\phi \in H^1(T)$; see \cite{acta25}. Since $\Gamma_u$ is covered by cut elements, \eqref{eq:cut-trace} followed by a standard inverse inequality~\cite{BrennerScott} gives the upper bound in \eqref{eq:norm-equivalence-pressure}.
\begin{theorem}[Stability]\label{thm:stability}
  Assuming that the domain $\Omega$ enjoys elliptic regularity, the multilinear form $\tilde A_h$ in \eqref{eq:discrete-problem-operator} satisfies the discrete inf-sup condition
    \begin{align}
      \inf_{(u_h,p_h) \in \Vh \times \Qh} \sup_{(v_h,q_h) \in \Vh \times \Qh} \frac{\tilde A_h((u_h,p_h),(v_h,q_h))}{(\|u_h\|^2_{\mathrm{d},h} + \|p_h\|^2_{0,h})^{1/2} (\|v_h\|^2_{\mathrm{d},h} + \|q_h\|^2_{0,h})^{1/2}} \geq \tilde \beta, 
    \end{align} 
   where, in the case $\Gamma_p=\emptyset$, the pressure spaces in the infimum and supremum are replaced by $\Qh \cap L^2_0(\Omega)$. The constant $\tilde \beta > 0$ is independent of $h$ and the cut locations, provided $\gamma$ is large enough and $h$ is small enough.
\end{theorem}
\begin{proof}
The proof proceeds in two steps.  First, we establish a discrete inf-sup condition for the bilinear form $b_h$ appearing in \eqref{eq:b}. Using this result together with coercivity of $a_h$, we can show an inf-sup condition for the reduced multilinear form $A_h$. In the second step, we estimate the boundary contribution $(v_h \cdot n, p_h)_{\Gamma_{u}}$ in terms of the mesh-dependent norms of $v_h$ and $p_h$ and show that, for $\gamma$ sufficiently large, stability of the full multilinear form $\tilde A_h$ follows.

We begin by constructing, for a given $p_h \in \Qh$ (with $p_h \in \Qh \cap L^2_0(\Omega)$ when $\Gamma_p=\emptyset$), a function $v_h \in \Vh$ that guarantees the inf-sup condition for the bilinear form $b_h$. Let $\xi \in H^2(\Omega)$ solve
\begin{align}\label{eq:neumann-mixed}
-\Delta \xi &= p_h \quad \text{in } \Omega, \qquad \partial_n \xi = 0 \quad \text{on } \bouu, \qquad \xi = 0 \quad \text{on } \boup.\end{align}
Note that if $\Gamma_u = \partial \Omega$, we have $\int_\Omega p_h = 0$ and $\xi \in H^2(\Omega) \cap L^2_0(\Omega)$. By elliptic regularity, it holds that
$
\|\xi\|_{H^2(\Omega)} \lesssim \|p_h\|_{\Omega}.
$ 
Next, defining $v_\Omega \doteq \nabla \xi \in [H^1(\Omega)]^d$, we obtain
\begin{equation}
\|v_\Omega\|_{\dimnor{H^1(\Omega)}} \lesssim \|p_h\|_{\Omega}, \qquad \dv{v_\Omega}=-p_h \quad \text{in } \Omega.
\end{equation} 

{Next, let $S = \Omega^{\mathrm{ext}} \setminus \Omega$ and $\chi_{\domact}$ be the characteristic function of $\domact$.} By~\cite[Theorem~III.3.1]{Galdi} and the trace extension theorem, one can show that there exists $v_S \in [H^1(S)]^d$ such that
\begin{align}
\dv{v_S} = - \chi_{\Omega_h} p_h \quad \text{in } S, \qquad 
v_S = v_\Omega \quad \text{on } \partial \Omega,
\end{align}
and the stability estimate
\begin{equation}
\|v_S\|_{\dimnor{H^1(S)}} \lesssim \|\chi_{\Omega_h}p_h\|_{L^2(S)}+\|v_\Omega\|_{\dimnor{H^{1/2}(\partial\Omega)}}
\end{equation}
holds. Note that there are no conditions imposed on $\partial S \setminus \partial \Omega$ and hence there are no compatibility conditions. We can now define $v \in [H^1(\Omega^{\mathrm{ext}})]^d$ such that 
\begin{equation}
v =
v_\Omega  \quad \text{in } \Omega,\qquad v =
v_S  \quad \text{in } S.
\end{equation}
Then
\begin{equation}
\dv v = - \chi_{\Omega_h} p_h \quad \text{in } \,   \domact \subset \Omega^{\mathrm{ext}}=\Omega \cup S, \qquad
v \cdot n = 0 \quad \text{on } \, \Gamma_u,
\end{equation}
and the following estimate holds by the stability estimate for $v_S$, elliptic regularity for $v_\Omega$ and the trace theorem on $\partial \Omega$:
\begin{equation}
\|v\|_{\dimnor{H^1(\Omega^{\mathrm{ext}})}} \lesssim  \|p_h\|_{L^2(\Omega)}+ \|\chi_{\Omega_h}p_h\|_{L^2(S)} +\|v_\Omega\|_{\dimnor{H^1(\Omega)}} \lesssim \|p_h\|_{L^2(\domact)}.
\end{equation}
Taking $v_h \doteq \Iv(v)$ and invoking the commutativity properties of the interpolator in Assumption~\ref{as:discrete-spaces}, we have that $v_h$ satisfies $\dv{v_h} = -p_h$ on $\domact$ and $\|v_h\|_{H(\mathrm{div},\Omega_h)} \lesssim \|p_h\|_{L^2(\domact)}$. Thus, we obtain:
\begin{align}
  \Bh(v_{h},p_h) = -(\dv{v_{h}},p_h)_{\dom} - \xiqh \shpf (\dv{v_{h}},p_h) =
  \| p_h \|^2_{L^2(\Omega)} + \xiqh \shpf(p_h,p_h) \gtrsim \| p_h \|^2_{L^2(\domact)},
\end{align}
where we have used \eqref{eq:stab-assump-2} in Assumption~\ref{as:stabilisation-terms}. By \eqref{eq:norm-equivalence-pressure}, the norms $\| p_h \|^2_{\domact}$ and $\|p_h\|^2_{0,h}$ are equivalent, and
 \begin{align}
    \| v_h \|^2_{\mathrm{d},h} & =  \|v_h\|^2_{H(\mathrm{div},\domact)} +\gamma  h^{-1} \|v_h \cdot n\|^2_{L^2(\bouu)} \\
    &  \lesssim \|p_h\|_{L^2(\domact)}^2+\gamma  h^{-1} \|(v-v_h) \cdot n\|^2_{L^2(\bouu)} \lesssim \|p_h\|_{L^2(\domact)}^2,  
\end{align}
where in the last inequality we used the properties of $v$, the trace inequality \eqref{eq:cut-trace} and the local interpolation estimate for $\Iv$ to obtain
\begin{align*}
h^{-1}  \|(v-v_h) \cdot n\|^2_{L^2(\bouu)}&\lesssim \sum_{T \in \mathcal{T}_{h}^{\mathrm{cut}} }  h^{-2} \|(v-v_h)\|^2_{L^2(T)}+ \|\nabla (v-v_h)\|^2_{\dimnor{L^2(T)}} \\
&  \lesssim \sum_{T \in \mathcal{T}_{h}} \|v\|_{H^1(T)}^2  \lesssim \|p_h\|_{L^2(\domact)}^2. 
\end{align*}
Next, we prove the stability of $a_h$ on $\mathrm{ker}(\Bh)$. We have, for any $u_h \in \Vh$:
\begin{align}\label{eq:stability-ah-basic}
  a_h(u_h,u_h)
  &= (u_h,u_h)_{\Omega} + \xivh \shuf(u_h,u_h) + \gamma  h^{-1} \|u_h \cdot n\|^2_{L^2(\bouu)} \\
  &\simeq  \|u_h\|_{\dimnor{L^2(\Omega_h)}}^2 + \gamma  h^{-1} \|u_h \cdot n\|^2_{L^2(\bouu)},
\end{align}
where we have used \eqref{eq:stab-assump-1} in Assumption~\ref{as:stabilisation-terms}. Next, we note that if $u_h \in \mathrm{ker} (\Bh)$, then
\[
0=\Bh(u_h,q_h)=-(\dv{u_h},q_h)_\Omega-\xiqh \shpf(\dv{u_h},q_h), \qquad \forall q_h \in \Qh.
\]
If $\Gamma_p \neq \emptyset$, then $\dv{u_h} \in \Qh$ and uniqueness of the stabilised projection problem \eqref{eq:stabilised-projection} with vanishing right-hand side implies that $\dv{u_h}=0$ on $\domact$. If $\Gamma_p=\emptyset$, then the identity above holds for all $q_h \in \Qh \cap L^2_0(\Omega)$. Since $\dv{u_h} \in \Qh$, we write $\dv{u_h}=w_h+c$ with $w_h \in \Qh \cap L^2_0(\Omega)$ and $c$ constant. Using that $\shpf(c,q_h)=0$ for all $q_h \in \Qh$, the same uniqueness argument yields $w_h=0$, and therefore $\dv{u_h}=c$ on $\domact$. In addition, since $\Gamma_u=\partial\Omega$ in this case, the divergence theorem gives $|\Omega|c=\int_\Omega \dv{u_h}=\int_{\Gamma_u} u_h\cdot n$. Hence, by Cauchy-Schwarz, we obtain,
\begin{equation*}
  \|\dv{u_h}\|_{L^2(\Omega_h)}^2
  = |\Omega_h|c^2
  = \frac{|\Omega_h|}{|\Omega|^2}\left(\int_{\Gamma_u} u_h\cdot n\right)^2
  \le \frac{|\Omega_h||\Gamma_u|}{|\Omega|^2}\|u_h \cdot n\|_{L^2(\Gamma_u)}^2
  \lesssim h^{-1}\|u_h \cdot n\|_{L^2(\Gamma_u)}^2,
\end{equation*}
where we have used that $|\Omega_h| \lesssim |\Omega|$. Combining this bound with \eqref{eq:stability-ah-basic}, we obtain the following stability result:
\begin{align}\label{eq:stability-ah}
  a_h(u_h,u_h) \simeq  \| u_h \|^2_{\mathrm{d},h},\quad \forall u_h \in \mathrm{ker}(\Bh).
\end{align}
Thus, using the standard Babu\v{s}ka--Brezzi theory \cite{Boffi2013-ge}, we obtain the global inf-sup condition:
\begin{align}\label{eq:stability-ah-inf-sup}
  \inf_{(u_h,p_h) \in \Vh \times \Qh} \sup_{(v_h,q_h) \in \Vh \times \Qh} \frac{A_h((u_h,p_h),(v_h,q_h))}{(\|u_h\|^2_{\mathrm{d},h} + \|p_h\|^2_{0,h})^{1/2} (\|v_h\|^2_{\mathrm{d},h} + \|q_h\|^2_{0,h})^{1/2}} \geq \beta, 
\end{align}
for some constant $\beta > 0$ independent of $h$ and the cut configuration. When $\Gamma_p=\emptyset$, the pressure spaces in \eqref{eq:stability-ah-inf-sup} are replaced by $\Qh \cap L^2_0(\Omega)$.

Next, we consider the full multilinear form $\tilde A_h$. We have 
\begin{align}
  \tilde A_h((u_h,p_h),(v_h,q_h))  &= A_h((u_h,p_h),(v_h,q_h)) + (v_h \cdot n, p_h)_{\Gamma_{u}}.
\end{align}
By the global inf-sup condition for $A_h$, for every admissible pair $(u_h,p_h)$ there exists a pair $(v_h,q_h)$ such that
\begin{align*}
\beta^{-1} A_h((u_h,p_h),(v_h,q_h)) &\geq \|u_h\|^2_{\mathrm{d},h} + \|p_h\|^2_{0,h}, \quad \| v_h \|^2_{\mathrm{d},h} +   \| q_h \|^2_{0,h} = \|u_h\|^2_{\mathrm{d},h}+ \|p_h\|^2_{0,h}.
\end{align*} 

We bound the boundary term using Young's inequality and the definition of the norms as follows:
\begin{align}
  (v_h \cdot n, p_h)_{\Gamma_{u}} 
  & \leq \alpha \gamma h^{-1}\| v_h \cdot n\|^2_{L^2(\bouu)} + \frac{h}{4 \alpha \gamma} \| p_h \|^2_{L^2(\bouu)} \leq \alpha \| v_h \|^2_{\mathrm{d},h} + \frac{1}{4 \alpha \gamma} \| p_h \|^2_{0,h}.
\end{align}
Combining the estimates yields
\begin{align*}
  \tilde A_h((u_h,p_h),(v_h,q_h)) &\geq \beta \|u_h\|^2_{\mathrm{d},h} + \beta \|p_h\|^2_{0,h} - \alpha (\| u_h \|^2_{\mathrm{d},h} +\|p_h\|^2_{0,h}) - \frac{1}{4 \alpha \gamma} \| p_h \|^2_{0,h}\\
  &\geq  (\beta-\alpha) \|u_h\|^2_{\mathrm{d},h} +  \left(\beta - \alpha -\frac{1}{4 \alpha \gamma}\right) \|p_h\|^2_{0,h}.
\end{align*}
Choosing $\alpha < \beta$ and assuming $\gamma$ is sufficiently large so that $\left(\beta - \alpha -\frac{1}{4 \alpha \gamma}\right) > 0$, we obtain the discrete inf-sup condition for $\tilde A_h$. For instance, taking $\alpha=\frac{\beta}{2}$ and $\gamma = \frac{2}{\beta^2}$ yields the inf-sup constant $\tilde{\beta}=\beta/4$.
\end{proof}
\section{A priori error analysis}\label{sec:error-estimates}
In this section, we prove a priori error estimates for the discrete problem \eqref{eq:discrete-problem}. The main result is stated in the following theorem, which provides an error bound in terms of the best-approximation errors in the mesh-dependent norms.
{
\begin{lemma}\label{as:extensions}
  Given $u \in [H^r(\dom)]^d$ and $p \in H^t(\dom)$, there exist extensions to $\domext$, still denoted by the same symbols, such that  
  $
    \| u \|_{\dimnor{H^{r}(\domext)}}   \lesssim
    \| u \|_{\dimnor{H^{r}(\Omega)}}$ and  
    $\| p \|_{H^{t}(\domext)} \lesssim \| p \|_{H^{t}(\Omega)}.
  $
\end{lemma}
\begin{proof}
   {The result is a consequence of Stein extension theorem (see  \cite[Ch.6, Thm.5]{Stein-Other-1971a}), which holds for any Lipschitz domain $\Omega$.}
\end{proof}
}
\begin{lemma}\label{as:extensions}
  Given $u \in [H^r(\domext)]^d$ and $p \in H^t(\domext)$ admit extensions to $\domext$, still denoted by the same symbols, such that $u \in [H^r(\domext)]^d$, $p \in H^t(\domext)$ and $\dv{u} \in H^s(\domext)$ for $0 \leq r \leq k_u+1$, $0 \leq s \leq k_p+1$ and $0 \leq t \leq k_p+1$, and these extensions satisfy the stability bound
  \begin{align*}
    \| u \|_{\dimnor{H^{r}(\domext)}}   \lesssim
    \| u \|_{\dimnor{H^{r}(\Omega)}}, \qquad \| \dv{u} \|_{H^{s}(\domext)} \lesssim \| u \|_{H^{s+1}(\Omega)}, \qquad \| p \|_{H^{t}(\domext)} \lesssim \| p \|_{H^{t}(\Omega)}.
  \end{align*}
\end{lemma}
\begin{proof}
   {The result is a consequence of Stein extension theorem (see  \cite[Ch.6, Thm.5]{Stein-Other-1971a}), which holds for any Lipschitz domain $\Omega$.}
\end{proof}

 and these extensions satisfy the stability bound

\begin{theorem}[Error estimate]
  Let $(u,p)$ be the solution of the continuous problem \eqref{eq:darcy} and let $(u_h,p_h)$ be the solution of the discrete problem \eqref{eq:discrete-problem}. {Let $u \in [H^{\max(k_u+1,k_p+2)}(\domext)]^d$ and $p \in H^{k_p+1}(\domext)$.}
  Using the conditions of Theorem~\ref{thm:stability}, the following error estimate holds:
  \begin{align}
    \|u - u_h\|_{\mathrm{d},h}^2 + \|p - p_h\|_{L^2(\Omega_h)}^2 \lesssim &
    h^{2{(k_u + \alpha)}} \| u \|^2_{\dimnor{H^{k_u+1}(\Omega)}} + h^{2(k_p+1)} {\| {u} \|^2_{H^{k_p+2}(\Omega)}} \\ &+ h^{2(k_p+1)} \| p \|^2_{H^{k_p+1}(\Omega)},
  \end{align}
  with $\alpha = 1$ if $\Gamma_u = \emptyset$ and $\alpha = 0$ otherwise.    
\end{theorem}
\begin{proof}
We prove the result in the case $\Gamma_p \neq \emptyset$. When $\Gamma_p=\emptyset$, the pressure quantities above are understood in $\Qh \cap L^2_0(\Omega)$, and the proof is analogous to the one below.
Let $(u,p)$ be the solution of the continuous problem~\eqref{eq:darcy} and let $(u_h,p_h)$ be the solution of the discrete problem~\eqref{eq:discrete-problem}. Let $(w_h,r_h) \in  \Vh \times  \Qh$. Using the expression of the discrete problem~\eqref{eq:discrete-problem} and the stability estimate in Theorem~\ref{thm:stability}, there exists $(v_h,q_h) \in \Vh \times \Qh$ such that
\begin{align}
& \| u_h - w_h \|^2_{\mathrm{d},h} + \| p_h - r_h \|^2_{0,h} \lesssim a_h(u_h - w_h,v_h) + \tilde b_h(v_h,p_h - r_h) + b_h(u_h - w_h,q_h) \nonumber \\
& = \ell_h(v_h,q_h) - a_h(w_h,v_h) - \tilde b_h(v_h,r_h) - b_h(w_h,q_h) \doteq (*). \label{eq:error-estimate-1}
\end{align}
Moreover, the pair $(v_h,q_h)$ can be chosen so that $\| v_h \|_{\mathrm{d},h} + \|q_h \|_{0,h} \lesssim \| u_h - w_h \|_{\mathrm{d},h} + \| p_h - r_h \|_{0,h}$.
{Using the definition of $\ell_h$ in \eqref{eq:discrete-problem-operator}, the strong form of \eqref{eq:darcy}, integration by parts and the boundary conditions, we obtain
\begin{align}
\ell_h(v_h,q_h)
&= (f,v_h)_{\Omega} + \gamma h^{-1} (u_\Gamma,v_h \cdot n )_{\bouu} - (v_h \cdot n,p_\Gamma)_{\boup} + (g,q_h)_{\Omega} \nonumber \\
&= (\eta u + \nabla p, v_h)_{\Omega} + \gamma h^{-1} (u_\Gamma,v_h \cdot n )_{\bouu} - (v_h \cdot n,p_\Gamma)_{\boup} + (g,q_h)_{\Omega} \nonumber \\
&= a(u,v_h) + \gamma h^{-1}(u \cdot n, v_h \cdot n)_{\bouu} + \tilde b(v_h,p) + b(u,q_h). \label{eq:error-estimate-3}
\end{align}}
Invoking \eqref{eq:error-estimate-3} in~\eqref{eq:error-estimate-1}, we can express the right-hand side of~\eqref{eq:error-estimate-1} as:
\begin{align} \label{eq:error-estimate-4}
  (*) = & a(u-w_h,v_h) + \gamma h^{-1}((u-w_h) \cdot n, v_h \cdot n)_{\bouu} + \tilde b(v_h,p-r_h) + b(u-w_h,q_h) \\ & - \xivh \shuf{(w_h,v_h)} + \xiqh \shpf{(\dv{v_h},r_h)} + \xiqh \shpf{(\dv{w_h},q_h)}.\nonumber
\end{align}
The first four terms on the right-hand side of~\eqref{eq:error-estimate-4} can be bounded using the Cauchy--Schwarz and Young inequalities together with the definition of the discrete norms:
\begin{align}
  & a(u-w_h,v_h) + \gamma h^{-1}((u-w_h) \cdot n, v_h \cdot n)_{\bouu} + \tilde b(v_h,p-r_h) + b(u-w_h,q_h) \\
& \lesssim ( \| u - w_h \|_{\mathrm{d},h} + \|p - r_h \|_{0,h} ) ( \| v_h \|_{\mathrm{d},h} + \|q_h \|_{0,h} ).
\end{align}
For the stabilisation terms, we proceed as follows:
\begin{align}
   -& \xivh \shuf{(w_h,v_h)} + \xiqh \shpf{(\dv{v_h},r_h)} + \xiqh \shpf{(\dv{w_h},q_h)} \\
   =& \xivh \shuf{(\Iv(u-w_h),v_h)} - \xiqh \shpf{(\dv{v_h},\Iq(p-r_h))} - \xiqh \shpf{(\dv{\Iv(u-w_h)},q_h)} \\ & - \xivh \shuf{(\Iv(u),v_h)} + \xiqh \shpf{(\dv{v_h},\Iq(p))} + \xiqh \shpf{(\dv{\Iv(u)},q_h)}. \label{eq:error-estimate-5}
\end{align}
The first three terms can readily be bounded by the continuity of the stabilisation bilinear forms and the stability of the interpolants $\Iv$ and $\Iq$:
\begin{align}
& \xivh \shuf{(\Iv(u-w_h),v_h)} - \xiqh \shpf{(\dv{v_h},\Iq(p-r_h))} - \xiqh \shpf{(\dv{\Iv(u-w_h)},q_h)} \\
& \lesssim ( \| u - w_h \|_{\mathrm{d},h} + \|p - r_h \|_{0,h} ) ( \| v_h \|_{\mathrm{d},h} + \|q_h \|_{0,h} ).
\end{align}
{In the following, let us use the compact notation $r = k_u+1$ and $s=t=k_p+1$.} 
The last three terms in~\eqref{eq:error-estimate-5} can be bounded by the weak consistency of the stabilisation bilinear forms in Assumption~\ref{as:stabilisation-terms}:
\begin{align*}
- \xivh &\shuf{(\Iv(u),v_h)} + \xiqh \shpf{(\dv{v_h},\Iq(p))} + \xiqh \shpf{(\dv{\Iv(u)},q_h)} \\
& \lesssim (h^{r} \| u \|_{\dimnor{H^{r}(\Omega)}} + {h^{s} \| {u} \|_{H^{s+1}(\Omega)}} + h^{t} \| p \|_{H^{t}(\Omega)})( \| v_h \|_{\mathrm{d},h} + \|q_h \|_{0,h} ),
\end{align*}
where, for the last term, we have used the commutativity property $\dv{\Iv(u)}=\Iq(\dv{u})$ together with the weak consistency of $\shpf$ applied to $\dv{u}$.
Using the estimate for $(v_h,q_h)$ stated after \eqref{eq:error-estimate-1}, we obtain
\begin{align*}
&  \| u_h - w_h \|_{\mathrm{d},h} + \| p_h - r_h \|_{0,h} \\
&  \lesssim
  \| u - w_h \|_{\mathrm{d},h} + \| p - r_h \|_{0,h}
  + h^{r} \| u \|_{\dimnor{H^{r}(\Omega)}} + {h^{s} \|  {u} \|_{H^{s+1}(\Omega)}} + h^{t} \| p \|_{H^{t}(\Omega)},
\end{align*}
after dividing by $\| u_h - w_h \|_{\mathrm{d},h} + \| p_h - r_h \|_{0,h}$ when this quantity is nonzero. {Choosing $w_h=\Iv(u)$ and $r_h=\Iq(p)$, the triangle inequality yields}
\begin{align*}
  {\|u-u_h\|_{\mathrm{d},h}} &{\leq \|u-\Iv(u)\|_{\mathrm{d},h} + \|u_h-\Iv(u)\|_{\mathrm{d},h},} \\
  {\|p-p_h\|_{0,h}} &{\leq \|p-\Iq(p)\|_{0,h} + \|p_h-\Iq(p)\|_{0,h}.}
\end{align*}
{The interpolation estimates for $u$, $\dv{u}$ and $p$, together with {the trace inequality for the boundary term in $\|\cdot\|_{\mathrm{d},h}$ and} Lemma~\ref{as:extensions}, then prove the theorem. In the case $\Gamma_{u} = \emptyset$, the trace inequality is not required.}
\end{proof}

{
\begin{remark}
We note the estimate is suboptimal for $k_u = k_p$ (the case of trimmed div-conforming spaces, i.e., Raviart-Thomas) for $\Gamma_{u} \neq \emptyset$, due to the interpolation error of the flux boundary penalty term. The result is always optimal for complete div-conforming polynomial spaces, i.e., BDM finite element spaces, since $k_u = k_p+1$. We note that optimal results can always be obtained in the body-fitted case because one can enforce the interpolant to preserve null traces. We have not observed this sub-optimality in the numerical experiments in Section~\ref{sec:num-results}.
\end{remark}
}

\subsection*{Pressure-robust flux error estimate}
We can improve the error estimate for the flux using ideas similar to those used in divergence-free mixed \ac{fem} for the Stokes system~\cite{John2017-hu}. However, since we consider the weak imposition of the flux boundary conditions, the results only hold for $\Gamma_{u} = \emptyset$; we note that we cannot enforce $w_h \cdot n = 0$ on $\Gamma_{u}$ since $\Gamma_{u}$ is embedded in $\Omega_h$. We define the space $\Vh(g) = \{ v_h \in \Vh \ : \ \dv{v_h} = -\pi_{h}^{0}(g) \}$. In the following, we use Lemma~\ref{as:extensions} and write $\tilde g$ for the extension of $-\dv{u}$ to $\domext$, so that $\tilde g=g$ in $\Omega$.

\begin{assumption}\label{as:pressure-kernel-approx}
  {We assume that the kernel of the pressure stabilisation, $\mathcal{K}_{h}^{p} \doteq \{ q_h \in \Qh \ : \ \shpf(q_h,\phi_h)=0 \ \forall \phi_h \in \Qh \}$, 
  has the following approximation property: for every $\tilde r \in H^t(\domext)$ with $0 \leq t \leq k_p+1$, there exists $w_h \in \mathcal{K}_{h}^{p}$ such that $\|\tilde r-w_h\|_{L^2(\Omega_h)}
    \lesssim h^t \|\tilde r\|_{H^t(\domext)}$.}
\end{assumption}

\begin{lemma}[Approximation of the stabilised projection]\label{lem:stabilised-projection-alt}
  Let $\tilde r \in H^t(\domext)$ with $0 \leq t \leq k_p+1$, and let $r$ denote its restriction to $\Omega$. Then, the stabilised projection $\pi_h^0(r) \in \Qh$ defined in \eqref{eq:stabilised-projection} satisfies
  \begin{equation}
    \|\tilde r-\pi_h^0(r)\|_{L^2(\Omega_h)} \lesssim h^t \|\tilde r\|_{H^t(\domext)}.
  \end{equation}
\end{lemma}
\begin{proof}
  {We consider $w_h \in \mathcal{K}_{h}^{p}$. Using the definition of the stabilised projection and the fact that $\shpf(w_h,q_h)=0$ for all $q_h \in \Qh$, we get}
  \begin{align*}
    (\pi_h^0(r)-w_h,q_h)_{\Omega} + \xiqh \shpf(\pi_h^0(r)-w_h,q_h) = (r-w_h,q_h)_{\Omega},
    \qquad \forall q_h \in \Qh.
  \end{align*}
  Taking $q_h=\pi_h^0(r)-w_h$ and using \eqref{eq:stab-assump-2}, we obtain
  \begin{align*}
    \|\pi_h^0(r)-w_h\|_{L^2(\Omega_h)}^2
    &\lesssim
    (r-w_h,\pi_h^0(r)-w_h)_{\Omega} \\
    &\le
    \|r-w_h\|_{L^2(\Omega)}\|\pi_h^0(r)-w_h\|_{L^2(\Omega)}
    \lesssim
    \|r-w_h\|_{L^2(\Omega)}\|\pi_h^0(r)-w_h\|_{L^2(\Omega_h)}.
  \end{align*}
  Hence,
  \begin{equation*}
    \|\pi_h^0(r)-w_h\|_{L^2(\Omega_h)} \lesssim \|r-w_h\|_{L^2(\Omega)}.
  \end{equation*}
  {By the triangle inequality,}
  \begin{equation*}
    \|\tilde r-\pi_h^0(r)\|_{L^2(\Omega_h)}
    \le
    \|\tilde r-w_h\|_{L^2(\Omega_h)} + \|\pi_h^0(r)-w_h\|_{L^2(\Omega_h)}
    \lesssim
    \|\tilde r-w_h\|_{L^2(\Omega_h)} + \|r-w_h\|_{L^2(\Omega)}.
  \end{equation*}
  {Invoking Assumption~\ref{as:pressure-kernel-approx}, we conclude that $\|\tilde r-\pi_h^0(r)\|_{L^2(\Omega_h)} \lesssim h^t \|\tilde r\|_{H^t(\domext)}$.}
\end{proof}

\begin{theorem}[Pressure-robust error estimate]\label{thm:pressure-robust-alt}
  {Let $u \in [H^{\max(k_u+1,k_p+2)}(\domext)]^d$ . Due to Lemma~\ref{as:extensions}, Assumption~\ref{as:pressure-kernel-approx}, and the conditions of Theorem~\ref{thm:stability}, the following error estimate holds for $\Gamma_{u} = \emptyset$:}
  \begin{align}
    & \|u - u_h\|_{\dimnor{L^2(\Omega)}} \lesssim h^{k_u+1} \| {u} \|_{H^{k_u+1}(\Omega)} + h^{k_p+1} \| {u} \|_{H^{k_p+2}(\Omega)}, \\
    & \| \dv{u} - \dv{u_h} \|_{L^2(\Omega)} \lesssim { h^{k_p+1} \| {u} \|_{H^{k_p+2}(\Omega)}}.
  \end{align}
\end{theorem}
\begin{proof}
For $\partial \Omega = \boup$, let $\tilde g \in H^s(\domext)$ be the stable extension of $-\dv{u}$ introduced above and define
\begin{equation*}
  \delta_h \doteq \Iq(\tilde g) - \pi_h^0(g) \in \Qh.
\end{equation*}
Using the discrete inf-sup construction in the proof of Theorem~\ref{thm:stability}, there exists $z_h \in \Vh$ such that
\begin{equation}\label{eq:divergence-correction}
  \dv{z_h} = \delta_h \quad \text{in } \Omega_h,
  \qquad
  \|z_h\|_{\mathrm{d},h} \lesssim \|\delta_h\|_{L^2(\Omega_h)}.
\end{equation}
Next, we define the divergence-corrected interpolant
\begin{equation*}
  w_h \doteq \Iv(u) + z_h.
\end{equation*}
Since $\Iv$ commutes with the divergence on $\domext$ and $\dv{u}=-\tilde g$ in $\domext$, we obtain
\begin{equation*}
  \dv{w_h}
  =
  \Iq(\dv{u}) + \dv{z_h}
  =
  -\Iq(\tilde g) + \dv{z_h}
  =
  -\pi_h^0(g),
\end{equation*}
and thus $w_h \in \Vh(g)$. In particular, $u_h-w_h \in \Vh(0)$.

Using the discrete problem~\eqref{eq:discrete-problem}, the continuous problem~\eqref{eq:weak-form}, the fact that $\Gamma_u=\emptyset$, and that $u_h-w_h \in \Vh(0)$, we obtain
\begin{align*}
  a_h(u_h-w_h,v_h)
  =
  (u-w_h,v_h)_{\Omega} - \xivh \shuf(w_h,v_h),
  \qquad \forall v_h \in \Vh(0).
\end{align*}
Taking $v_h=u_h-w_h \in \Vh(0)$ and using the coercivity of $a_h$ in $\Vh(0)$, we get
\begin{align}\label{eq:bound-uh-1}
  \|u_h-w_h\|_{\mathrm{d},h}^2
  &\lesssim
  (u-w_h,u_h-w_h)_{\Omega} - \xivh \shuf(w_h,u_h-w_h).
\end{align}
Since $w_h=\Iv(u)+z_h$, we can split the stabilisation term as
\begin{align}\label{eq:bound-uh-2}
  - \xivh \shuf(w_h,u_h-w_h)
  =
  - \xivh \shuf(\Iv(u),u_h-w_h) - \xivh \shuf(z_h,u_h-w_h).
\end{align}
Using the Cauchy--Schwarz inequality for the first term on the right-hand side of \eqref{eq:bound-uh-1} and the weak consistency of $\shuf$ on $\Iv(u)$ and the continuity of $\shuf$ in Assumption~\ref{as:stabilisation-terms} for the terms in \eqref{eq:bound-uh-2}, we obtain
\begin{align*}
  \|u_h-w_h\|_{\mathrm{d},h}^2
  &\lesssim
  \|u-w_h\|_{\dimnor{L^2(\Omega)}} \|u_h-w_h\|_{\dimnor{L^2(\Omega)}}
  + h^r \|u\|_{\dimnor{H^r(\Omega)}} \|u_h-w_h\|_{\mathrm{d},h}
  \\
  &\quad
  + \|z_h\|_{\mathrm{d},h}\|u_h-w_h\|_{\mathrm{d},h}.
\end{align*}
{Again, we use the compact notation $r = k_u+1$ and $s = k_p+1$.}
Dividing by $\|u_h-w_h\|_{\mathrm{d},h}$ and using that $\|u_h-w_h\|_{\dimnor{L^2(\Omega)}} \le \|u_h-w_h\|_{\mathrm{d},h}$, we infer that
\begin{align}\label{eq:alt-pressure-robust-bound-1}
  \|u_h-w_h\|_{\mathrm{d},h}
  &\lesssim
  \|u-w_h\|_{\dimnor{L^2(\Omega)}}
  + h^r \|u\|_{\dimnor{H^r(\Omega)}}
  + \|z_h\|_{\mathrm{d},h}.
\end{align}
Using the triangle inequality and the norm definitions, we obtain
\begin{align}\label{eq:alt-pressure-robust-bound-2}
  \|u-w_h\|_{\dimnor{L^2(\Omega)}} 
  & 
  \leq
  \|u-\Iv(u)\|_{\dimnor{L^2(\Omega)}} + \|z_h\|_{\dimnor{L^2(\Omega)}} \nonumber \\
  & \leq
  \|u-\Iv(u)\|_{\dimnor{L^2(\Omega)}} + \|z_h\|_{\mathrm{d},h}.
\end{align}
Combining the triangle inequality with \eqref{eq:alt-pressure-robust-bound-1} and \eqref{eq:alt-pressure-robust-bound-2}, we arrive at
\begin{equation}\label{eq:alt-pressure-robust-bound-3}
  \|u-u_h\|_{\dimnor{L^2(\Omega)}}
  \lesssim
  \|u-\Iv(u)\|_{\dimnor{L^2(\Omega)}}
  +
  \|z_h\|_{\mathrm{d},h}
  +
  h^r \|u\|_{\dimnor{H^r(\Omega)}}.
\end{equation}
Moreover,
\begin{align*}
  \|\delta_h\|_{L^2(\Omega_h)}
  &\leq
  \|\tilde g-\Iq(\tilde g)\|_{L^2(\Omega_h)}
  +
  \|\tilde g-\pi_h^0(g)\|_{L^2(\Omega_h)}
  \\
  &\lesssim
  h^s \|\tilde g\|_{H^s(\domext)}
  \lesssim
   { h^s \| {u} \|_{H^{s+1}(\Omega)}}.
\end{align*}
  {Here we have used the approximation properties of $\Iq$, Lemma~\ref{lem:stabilised-projection-alt} applied to the extension $\tilde g$ of $g$, and the stability of the extension.}
Using \eqref{eq:divergence-correction}, we get
\begin{equation*}
  \|z_h\|_{\mathrm{d},h}
  \lesssim
  { h^s \| {u} \|_{H^{s+1}(\Omega)}}.
\end{equation*}
Putting together the previous estimates with the approximation properties of $\Iv$ in Assumption~\ref{as:discrete-spaces} proves the first bound. For the second estimate, we use directly the identity below together with Lemma~\ref{lem:stabilised-projection-alt}.
\begin{equation*}
  \dv{u} - \dv{u_h} = -g + \pi_h^0(g) = -\tilde g + \pi_h^0(g)
\end{equation*}
in $\Omega$. Therefore,
\begin{equation*}
  \|\dv{u} - \dv{u_h}\|_{L^2(\Omega)}
  \le \|\tilde g - \pi_h^0(g)\|_{L^2(\Omega_h)}
  \lesssim h^s \|\tilde g\|_{H^s(\domext)}
  \lesssim  { h^s \| {u} \|_{H^{s+1}(\Omega)}},
\end{equation*}
which completes the proof.
\end{proof}

\begin{remark}[Pressure post-processing~\cite{Lehrenfeld-MathComput-2024j}]\label{rem:postprocessing}
{By~\cite[Lemma~1]{Lehrenfeld-MathComput-2024j}, when $\Gamma_u = \emptyset$ and the discrete load is $f_h = \pi_h^0(f)$, the flux $u_h$ of the present method coincides with that of~\cite{Lehrenfeld-MathComput-2024j} away from cut elements and their direct neighbours. Consequently, the element-local and patch-wise post-processings of~\cite[Theorems~4.1 and~4.2]{Lehrenfeld-MathComput-2024j} apply and yield a post-processed pressure $p_h^*$ converging at rate $O(h^{k_p+2})$ in $L^2(\Omega)$, one order above the direct estimate. 
}
\end{remark}

\begin{remark}
{
The body-fitted simplification of the proposed method, where the stabilisation terms vanish since $\Omega = \Omega_h$ and there are no cut cells, was presented in \cite{Burman-JNumerMath-2022z}. However, the authors only considered the numerical analysis of the symmetric variant (which destroys mass conservation). The stability and convergence analyses above readily extend to the body-fitted case, with one difference. One can use optimal interpolation properties in the discrete space of null traces so there is no optimality loss for trimmed polynomial spaces.}
\end{remark} 

\section{Augmented Lagrangian formulation}\label{sec:aug-lag}
In this section, we extend the discrete problem \eqref{eq:discrete-problem} using an augmented Lagrangian approach. When $\Gamma_p \neq \emptyset$, Lemma~\ref{lem:mass-conservation} implies that the solution $u_h \in \Vh$ of~\eqref{eq:discrete-problem} satisfies
\begin{align}\label{eq:aug-lag-formulation}
  c_{\mathrm{AL}}(u_h,v_h) = \ell_{\mathrm{AL}}(v_h),
\end{align}
with
\begin{align}
  c_{\mathrm{AL}}(u_h,v_h) \doteq (\dv{u_h},\dv{v_h})_{\Omega} + \xiqh \shpf(\dv{u_h},\dv{v_h}),  \qquad
  \ell_{\mathrm{AL}}(v_h)  \doteq -(g,\dv{v_h})_{\Omega}.
\end{align}
Accordingly, the augmented Lagrangian formulation of the discrete problem~\eqref{eq:discrete-problem} is defined as follows: find $(u_h,p_h) \in \Vh \times \Qh$ such that
\begin{align} \label{eq:aug-lag-form}
  \tilde A_h((u_h,p_h),(v_h,q_h)) + \tau_{\mathrm{AL}} c_{\mathrm{AL}}(u_h,v_h) = \ell_h(v_h,q_h) + \tau_{\mathrm{AL}}  \ell_{\mathrm{AL}}(v_h),
\end{align}
for all $(v_h,q_h) \in \Vh \times \Qh$, where $\tau_{\mathrm{AL}}>0$ is a user-defined parameter. 
{In the case $\Gamma_p=\emptyset$, the argument is slightly modified in view of the remark after Lemma~\ref{lem:mass-conservation}. Let $m_\Omega(r) \doteq |\Omega|^{-1}(r,1)_\Omega$ denote the mean value of $r \in L^2(\Omega)$. Since the conservation equation is tested only with functions in $\Qh \cap L^2_0(\Omega)$, taking $q_h=\dv{v_h}-m_\Omega(\dv{v_h}) \in \Qh \cap L^2_0(\Omega)$ shows that the solution $u_h$ satisfies \eqref{eq:aug-lag-formulation} with the following modified bilinear and linear forms:
\begin{align}
  c_{\mathrm{AL}}(u_h,v_h) &\doteq (\dv{u_h}-m_\Omega(\dv{u_h}),\dv{v_h})_{\Omega} + \xiqh \shpf(\dv{u_h},\dv{v_h}), \\
  \ell_{\mathrm{AL}}(v_h) &\doteq -(g-m_\Omega(g),\dv{v_h})_{\Omega}, \qquad \forall v_h \in \Vh.
\end{align}
In this case, the augmented Lagrangian formulation is then obtained by using $\Qh \cap L^2_0(\Omega)$ as pressure space in \eqref{eq:aug-lag-form}. 
}

In both cases, the solution of the augmented Lagrangian formulation coincides with the solution of the discrete problem~\eqref{eq:discrete-problem}. Moreover, this formulation is more amenable to operator preconditioning. The flux block is now elliptic in $H(\mathrm{div},\Omega)$ and the pressure Schur complement is spectrally equivalent to the mass matrix of the pressure space as $\tau_{\mathrm{AL}} \to \infty$. These blocks can be efficiently preconditioned using robust multigrid methods \cite{Arnold-NumerMathHeidelb-2000z}. We show experimentally the performance of the augmented Lagrangian formulation in Section~\ref{sec:num-results}, but we do not analyse solvers for this formulation in this work. We also note that the stability and convergence analysis in the previous sections can be readily extended to the augmented Lagrangian formulation, since the additional terms vanish for the solution of the discrete problem~\eqref{eq:discrete-problem}.

\section{Stabilisation} \label{sec:stab}
In this section, we discuss different stabilisation terms that satisfy Assumption~\ref{as:stabilisation-terms}.

\subsection{Cell-wise stabilisation}
Let us introduce a linear operator $\Hstr : \Vhc{\mathrm{*}}{} \to \Vhc{\mathrm{*}}{}$, for $\mathrm{*} \in \{\mathrm{d},0\}$. We propose the following stabilisation penalty terms:
\begingroup
\mathtoolsset{showonlyrefs=false}
\begin{subequations}\label{eq:stab-terms-cell}
\begin{align}
  \shuf(u,v) & \doteq  (u - \Hv(u),v - \Hv(v))_{{\Omega_h^{\rm cut}}}, \\
  \shpf(p,q) & \doteq (p - \Hq(p),q - \Hq(q))_{{\Omega_h^{\rm cut}}}, 
\end{align}
\end{subequations}
\endgroup
In this section, we discuss the choice of the $\Hstr$ operators used in the stabilisation terms. We consider two natural choices. Both definitions rely on the following discrete extension operator.
\subsection{Extension operator}
{We use the aggregated mesh $\Thag$ introduced in Section~\ref{sec:geom}, together with the associated root cell $T^{\mathrm{rt}}$ in each aggregate.}
In order to state (or analyse) the unfitted formulations proposed in this work, we define the extension operator $\extv : \Vh \to \Vh^-$, where 
\begin{equation}
  \Vh^{-} = \left\{ v_h \in H(\mathrm{div},\domin) : v_h|_T \in \mathcal{V}(T), \forall T \in \mathcal{T}_h^{\mathrm{ag}} \right\},
\end{equation}
i.e., a super-space of $\Vh$ functions that can have discontinuous normal components on $\mathcal{F}_{h}^{\mathrm{cut}}$.\footnote{The operator $\extv$ is a \emph{non-conforming} discrete extension operator; its image is not a subspace of $H(\mathrm{div},\Omega)$. This is different from the extension proposed in \cite{Badia2018-vw} for $H^1$-conforming spaces and used in \cite{Badia2022-zq} combined with ghost-penalty stabilisation.}
For each aggregate $T \in \mathcal{T}_{h}^{\mathrm{ag}}$, we identify the corresponding root cell $T^{\mathrm{rt}}$ and extract $v(T^{\mathrm{rt}})$, which is the restriction onto $T$ of a polynomial $p_{v}$ in $\mathcal{V}$. We thus define $\extv(v_h)|_T = p_{v}|_T$.
Therefore, $\extv$ extends the function aggregate-wise from the root cell to the cut cells in the aggregate. 
Taking $\Hv = \extv$ in the stabilisation terms, we obtain a method that is similar to the bulk ghost-penalty method proposed in~\cite{Badia2022-zq} for grad-conforming spaces. 
We proceed analogously for the pressure space, defining $\extq : \Qh \to \Qh$. However, since $\Qh$ is a discontinuous space, the extension still belongs to $\Qh$, due to Assumption~\ref{as:discrete-spaces}.

\subsection{Aggregate-wise projection}
For each aggregate $T \in \mathcal{T}_h^{\mathrm{ag}}$, we define $\Pv(v_h)|_T = \piulocal(v_h|_T)$, where $\piulocal$ is the $L^2$-projection onto the local flux space $\V(T)$:
\begin{alignat}{2}
  \piulocal(v_h) &\in \V(T) \ : \ 
  \int_T (v_h - \piulocal(v_h)) \cdot w &= 0, \qquad &\forall w \in \V(T).
\end{alignat}
If we take $\Hv = \Pv$, we recover a bulk ghost-penalty-like method~\cite{Burman2010-tf} on an agglomerated mesh~\cite{Badia2022-zq}. 

\subsection{Stability and approximability properties}
Next, we prove that the proposed stabilisation terms satisfy the stability and approximability requirements used in the abstract analysis above.
\begin{lemma}\label{lem:enhanced-stability}
  The stabilisation terms~\eqref{eq:stab-terms-cell} satisfy Assumption~\ref{as:stabilisation-terms}. Moreover, the pressure stabilisation satisfies Assumption~\ref{as:pressure-kernel-approx}.
\end{lemma}
\begin{proof}
First, we prove the enhanced stability results. For the pressure stabilisation, the proof for $\Hq = \extq$ is a direct consequence of the stability of the aggregate-wise extension~\cite{Badia2018-vw}, i.e., $\|\extq p_h \|_{\domact} \lesssim \|p_h\|_{\domin}$, the fact that $p_h - \extq p_h = 0$ in $\domin$ and the triangle inequality. The result for the $L^2$-projection, i.e., $\Hq = \Pq$, can be proved as in \cite{Burman2010-tf}, by mapping the aggregate onto an aggregate of unit diameter, checking that the right-hand side is a norm and using the equivalence of discrete norms together with a scaling argument. The corresponding estimate for the flux stabilisation follows by the same argument.

Next, we prove the weak consistency for $\shuf$. Let $w \in [H^{r}(\domext)]^d$. We note that the interpolant $\Iv$ can be defined in a cell-wise manner, i.e., $\Iv(w)|_T = \IvT(w|_T)$, where $\IvT$ is the local interpolant defined on the cell $T \in \Th$. We can also define $\IvT$ on aggregates $T \in \Thag$ by interpolating on the cells belonging to the aggregate. On the other hand, the operator $\Hv$ is already defined aggregate-wise, so one can write $\Hv(\Iv(w))|_T = \HvT(\Iv(w)|_T)$ for $T \in \Thag$. 
 We consider $\xi_h \in \mathbb{P}_{k_u}(\Thag)$; thus, $\HvT(\xi_h) = \xi_h$ for each aggregate $T \in \Thag$. We have:
\begin{align*}
  \shuf{(\Iv(w),v_h)} &= \sum_{T \in \Thag} (\IvT(w) - \HvT(\IvT(w)),v_h - \HvT(v_h))_T \\ 
  & = \sum_{T \in \Thag} (\IvT(w - \xi_h) - \HvT(\Iv(w-\xi_h)),v_h - \HvT(v_h))_T
  \\
  & \lesssim \sum_{T \in \Thag} \| w - \xi_h \|_{\dimnor{L^2(T)}} \| v_h \|_{\dimnor{L^2(T)}},
\end{align*}    
 for any $v_h \in \Vh$, 
where we have used the stability of $\Hv$ and $\Iv$ in the last inequality. Owing to a Bramble--Hilbert argument, we obtain the desired result. The bounds corresponding to $\shpf$ can be proved analogously, and the fact that the pressure stabilisation vanishes on constants is straightforward by definition of the extension.

{To prove Assumption~\ref{as:pressure-kernel-approx}, let $\tilde r \in H^t(\domext)$ with $0 \leq t \leq k_p+1$. By a Bramble--Hilbert argument on the aggregates, and using that the aggregate diameters are comparable to $h$, there exists $w_h \in \mathbb{P}_{k_p}(\Thag) \subset \Qh$ such that}
\begin{equation*}
{  \|\tilde r-w_h\|_{L^2(\Omega_h)} \lesssim h^t \|\tilde r\|_{H^t(\domext)}.}
\end{equation*}
{It remains to show that $w_h \in \mathcal{K}_h^p$. Since $w_h$ is polynomial on each aggregate, both choices $\Hq=\extq$ and $\Hq=\Pq$ reproduce $w_h$ exactly on every aggregate. Hence $w_h-\Hq(w_h)=0$ in $\Omega_h^{\mathrm{cut}}$, and for every $\phi_h \in \Qh$ we obtain}
\begin{equation*}
{  \shpf(w_h,\phi_h) = (w_h-\Hq(w_h),\phi_h-\Hq(\phi_h))_{\Omega_h^{\mathrm{cut}}} = 0.}
\end{equation*}
{Thus $w_h \in \mathcal{K}_h^p$, and Assumption~\ref{as:pressure-kernel-approx} follows.}
\end{proof}

\subsection{Face-based stabilisation}\label{sec:face-stab}
Recall that $\mathcal{F}_{h}^{\mathrm{cut}}$ contains all interior facets that belong to cut cells. A subset of these facets, denoted by $\Fstab$, is used in the following face-based stabilisation forms:
\begingroup
\mathtoolsset{showonlyrefs=false}
\begin{subequations}\label{eq:shface}
\begin{align}
\shuf(u,v) &= \sum_{F \in \Fstab}  \sum_{j=0}^{k_u}  h^{2j + 1}  ([D_n^j u],[D_n^j v])_F, \label{eq:shuf}\\
\shpf (p,q) &= \sum_{F \in \Fstab}  \sum_{j=0}^{k_p}  h^{2j + 1}  ([D^j p],[D^j q])_F, \label{eq:shpf}
\end{align}
 \end{subequations}
\endgroup
Here, $D^j$ denotes the generalized derivative of order $j$ and $\jump{D_n^j v}$ denotes the jump of the normal derivative of order $j$ across the face $F$, with $\jump{D_n^0 v}=\jump{v}$ and $n$ the unit normal associated with $F$. The full jump $\jump{D^j q}$ also includes jumps of derivatives in directions orthogonal to $n$. The constants $k_u$ and $k_p$ denote the degrees of the polynomials used in the discrete spaces $\Vh$ and $\Qh$, respectively. For an alternative implementation of the standard ghost-penalty form~\eqref{eq:shface}, which is particularly convenient for higher-order polynomial approximations, we refer to~\cite{preuss2018higher}. See also \cite[Chapter 4]{acta25} for other equivalent stabilisation terms.
{In the simplest case, $\Fstab = \mathcal{F}_{h}^{\mathrm{cut}}$. To reduce the amount of stabilisation, we instead define $\Fstab$ from the aggregated mesh introduced in Section~\ref{sec:geom}. For each non-trivial aggregate $A \in \Thagb$ with root cell $T^{\mathrm{rt}}$, we choose a connected set of facets $\mathcal{F}_A \subset \mathcal{F}_{h}^{\mathrm{cut}}$ such that every cell in $A$ is connected to $T^{\mathrm{rt}}$ through a path of facets in $\mathcal{F}_A$, and we set $\Fstab \doteq \bigcup_{A \in \Thagb} \mathcal{F}_A$. We refer to \cite[Algorithm~1 and Figure~1]{LarZah23} and \cite[Figure~4.1]{acta25} for illustrations.} We now show that Assumption~\ref{as:stabilisation-terms} holds.
\begin{lemma}\label{lem:enhanced-stability-FGP}
  The stabilisation terms~\eqref{eq:shface} satisfy Assumption~\ref{as:stabilisation-terms}. {Moreover, the pressure stabilisation satisfies Assumption~\ref{as:pressure-kernel-approx}.}
\end{lemma}
\begin{proof}
For $r_h \in\Qh$, the estimate $\| r_h \|_{L^2(\Omega_h)}^2 \lesssim \| r_h \|^2_{L^2(\Omega)} + \shpf(r_h,r_h)$ follows from a Taylor expansion argument for polynomial functions around points on connecting facets; see for example~\cite{acta25},~\cite[Lemma 3.8]{Hansbo2014-bd} and \cite[Lemma 5.1]{Massing-JSciComput-2014r}. 
For the reverse inequality, we use the standard element-wise trace inequality followed by the standard inverse inequality~\cite{BrennerScott} to obtain
	\begin{align} \label{eq:spterm}
		\shpf(r_h,r_h)  &
                \lesssim \sum_{T \in \mathcal{T}_{h}} \| r_h \|_{L^2(T)}^2.
	\end{align}
Consequently, $\| r_h \|^2_{L^2(\Omega)} + \shpf(r_h,r_h)\lesssim  \| r_h \|_{L^2(\Omega_h)}^2$. 
Let $r \in H^t(\Omega^{\mathrm{ext}})$ with $0 \le t \le k_p+1$ (integer), and let $q_h \in \Qh$.  
Applying the Cauchy--Schwarz inequality and~\eqref{eq:spterm} gives
\begin{align}
\shpf(\Iq(r),q_h) 
&\le \shpf(\Iq(r),\Iq(r))^{1/2} \, \shpf(q_h,q_h)^{1/2} 
\lesssim \shpf(\Iq(r),\Iq(r))^{1/2} \, \|q_h\|_{L^2(\Omega_h)}.
\end{align}
To bound $\shpf(\Iq(r),\Iq(r))$, consider an arbitrary $r_h \in \Qh$ and split the sum over derivatives into low- and high-order terms relative to the regularity $t$:
\begin{align}
\shpf(r_h,r_h) 
= \sum_{F \in \Fstab} \sum_{j=0}^{t-1} h^{2j+1} \| [D^j r_h] \|_{L^2(F)}^2
+ \sum_{F \in \Fstab} \sum_{j=t}^{k_p} h^{2j+1} \| [D^j r_h] \|_{L^2(F)}^2.
\end{align}
For the high-order terms $j \ge t$, apply standard trace and inverse inequalities~\cite{BrennerScott} to bound them in terms of $\| D^{t} r_h \|_{L^2(T)}$:
\begin{align}
\shpf(r_h,r_h) 
\lesssim \sum_{F \in \Fstab} \sum_{j=0}^{t-1} h^{2j+1} \| [D^j r_h] \|_{L^2(F)}^2
+ h^{2t} \sum_{T \in \mathcal{T}_h} \| D^t r_h \|_{L^2(T)}^2.
\end{align}
Now set $r_h = \Iq(r)$ and subtract $r$ for the low-order terms and use that 
\begin{equation}
[D^j \Iq(r)] = [D^j (\Iq(r)-r)] + [D^j r] = [D^j (\Iq(r)-r)], \quad j < t,
\end{equation}
since $r \in H^t(\Omega^{\mathrm{ext}})$ implies $[D^j r] = 0$ for $j=0,\cdots,t -1$.
Using the standard trace inequality on each face $F$:
\begin{equation}
\| [D^j (\Iq(r)-r)] \|_{L^2(F)}^2 
\lesssim h^{-1} \| D^j (\Iq(r)-r) \|_{L^2(T)}^2 + h \| D^{j+1} (\Iq(r)-r) \|_{L^2(T)}^2
\end{equation}
and standard interpolation estimates
\begin{equation}
\| D^j (\Iq(r)-r) \|_{L^2(T)} \lesssim h^{t-j} \, |r|_{H^t(T)}, 
\qquad
\| D^{j+1} (\Iq(r)-r) \|_{L^2(T)} \lesssim h^{t-j-1} \, |r|_{H^t(T)},
\end{equation}
we readily obtain
\begin{align}
h^{2j+1} \| [D^j (\Iq(r)-r)] \|_{L^2(F)}^2 \lesssim h^{2t} |r|_{H^t(T)}^2.
\end{align}
Summing over all faces $F \in \Fstab$ and adding the high-order term gives
\begin{align}
\shpf(\Iq(r),\Iq(r)) 
&\lesssim h^{2t} \| r \|_{H^t(\Omega^{\mathrm{ext}})}^2 
+ h^{2t} \sum_{T \in \mathcal{T}_h} \| D^t \Iq(r) \|_{L^2(T)}^2
\lesssim h^{2t} \| r \|_{H^t(\Omega^{\mathrm{ext}})}^2,
\end{align}
where in the last inequality we used the stability of the interpolant
\begin{equation}
\sum_{T \in \mathcal{T}_h} \| D^t \Iq(r) \|_{L^2(T)}^2 
\lesssim \sum_{T \in \mathcal{T}_h} \| D^t (\Iq(r)-r) \|_{L^2(T)}^2 + \| D^t r \|_{L^2(T)}^2
\lesssim \| r \|_{H^t(\Omega^{\mathrm{ext}})}^2.
\end{equation}
Hence we conclude the weak consistency bound:
\begin{align}
\shpf(\Iq(r),q_h) \lesssim h^t \| r \|_{H^t(\Omega^{\mathrm{ext}})} \, \| q_h \|_{L^2(\Omega_h)}, 
\qquad 0 \le t \le k_p+1.
\end{align}
Finally, $\shpf(c,q_h)=0$ for every constant $c$ and every $q_h \in \Qh$, since all jumps of a constant and of its derivatives vanish identically.
{The proof of Assumption~\ref{as:pressure-kernel-approx} follows the same lines as above. Let $\tilde r \in H^t(\domext)$ with $0 \leq t \leq k_p+1$. Since the patches induced by $\Fstab$ are the aggregates in $\Thagb$, a Bramble--Hilbert argument on these aggregates yields $w_h \in \Qh$, polynomial of degree at most $k_p$ on each patch, such that $\|\tilde r-w_h\|_{L^2(\Omega_h)} \lesssim h^t \|\tilde r\|_{H^t(\domext)}$. Since $w_h$ is given on each patch by a single polynomial, all jumps $[D^j w_h]$ vanish on every facet in $\Fstab$ for $j=0,\dots,k_p$, and therefore $\shpf(w_h,\phi_h)=0$ for every $\phi_h \in \Qh$. Hence $w_h \in \mathcal{K}_h^p$, and Assumption~\ref{as:pressure-kernel-approx} follows.}
The corresponding estimate for $\uh$ follows by the same argument, applied component-wise to each component of $\uh$. 
\end{proof}
\section{Numerical experiments}\label{sec:num-results}
We present a set of numerical experiments to illustrate the performance of the proposed stabilised unfitted mixed finite element methods for the Darcy problem. We consider bulk-based and face-based stabilisation techniques, as introduced in Section~\ref{sec:stab}, and study their behaviour for different cut configurations, boundary conditions and numerical parameter regimes.

\subsection{Bulk-based stabilisation methods} \label{subsec:num-results-bulk}
In this subsection, we compare the unstabilised method (\texttt{std}) with the bulk ghost-penalty stabilised method (\texttt{BGP}), which relies on the cell-wise stabilisation terms \eqref{eq:stab-terms-cell} in combination with the aggregate-wise $L^2$-projection operators $\Pv$ and $\Pq$. We also consider the augmented Lagrangian formulation \eqref{eq:aug-lag-form} of the latter stabilisation method (\texttt{AL-BGP}). Table~\ref{tab:methods} summarises the different numerical methods considered in this subsection.
\begin{table}[h] \caption{Summary of the bulk-stabilisation methods.} \label{tab:methods}
\begin{tabular}{@{}lllll@{}}
\toprule
Method & Discrete Problem & Stabilisation terms & Stabilisation type & Symbol \\ \midrule
\texttt{std} & Problem \eqref{eq:discrete-problem} & None ($\tau_{\bullet}=0$) & None & \hspace{12pt} \includegraphics[trim=10pt 13pt 10pt 15pt,clip]{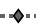} \\ 
\texttt{BGP} & Problem \eqref{eq:discrete-problem} & Bulk stabilisation \eqref{eq:stab-terms-cell} & $\Hv=\Pv,\Hq=\Pq$ & \hspace{12pt} \includegraphics[trim=10pt 13pt 10pt 15pt,clip]{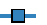}        \\
\texttt{AL-BGP} & Problem \eqref{eq:aug-lag-form} & Bulk stabilisation \eqref{eq:stab-terms-cell} & $\Hv=\Pv,\Hq=\Pq$ & \hspace{12pt} \includegraphics[trim=10pt 13pt 10pt 15pt,clip]{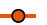} \\ 
\bottomrule
\end{tabular}
\end{table}
We consider a square physical domain $\Omega = [-h_{\mathrm{cut}}-L/2,L/2+h_{\mathrm{cut}}]^2$ embedded in the background domain $\Omega_h^{\mathrm{bg}} = [-h-L/2,h+L/2]^2$, consisting of a uniform Cartesian mesh with $n\times n$ quadrilateral cells of size $h = L/(n-2)$. Here, $h_{\mathrm{cut}} \in (0,h)$ is the parameter controlling the cut-cell size and the inner-domain length is fixed at $L=1.0$. To add the cell-wise stabilisation for the \texttt{BGP} and \texttt{AL-BGP} methods, we create an aggregated mesh $\mathcal{T}_h^{\mathrm{ag}}$ using the aggregation parameter $\delta = 1.0$.
For the discrete spaces, 
we use the pair $\mathbb{RT}_0 \times \mathbb{Q}_0$. We consider pure pressure boundary conditions ($\Gamma = \boup$), pure flux boundary conditions ($\Gamma = \bouu$) and mixed boundary conditions. This allows us to analyse the behaviour of the methods with respect to the weak imposition of the flux boundary conditions and its numerical penalty parameter $\gamma$ separately. Recall that for pure flux boundary conditions ($\Gamma = \bouu$), the discrete pressure space is $\Qh \cap L^2(\Omega)$.

All numerical results presented in this subsection have been obtained using the open-source numerical framework Gridap \cite{Badia2020,Verdugo2022} and its sub-package GridapEmbedded \cite{GridapEmbeddedRepo}, both written in \texttt{Julia}. The software (available in \cite{PaperUnfittedFEMDarcy}) was run on the supercomputer Gadi, hosted by the Australian National Computational Infrastructure Agency (NCI). For efficiency reasons, the condition numbers reported in this subsection have been computed in the 1-norm using \texttt{Julia}'s \texttt{cond()} method.

\subsubsection{$h$-convergence on a cut square}
To study the $h$-convergence of the different stabilisation methods, we vary the mesh refinement level via $n$ while keeping the cut-cell length ratio $h_{\mathrm{cut}}/h$ constant. We consider two different cut configurations: a \emph{large cut} configuration with $h_{\mathrm{cut}}/h =5.0\cdot10^{-1}$ and a \emph{small cut} configuration with $h_{\mathrm{cut}}/h = 5.0\cdot10^{-7}$. To compute the errors analytically, we employ the following 2-D manufactured solution:
\begin{equation} \label{eq:manufactured.sol.smooth}
    p(x,y) = \sin(\pi x)-\sin(\pi y), \text{ and } {u}(x,y) = (x+\sin(\pi y),-y + \sin(\pi x)).
\end{equation}
In Figure~\ref{fig:hconvergence}, we show that, for both cut configurations, the flux and pressure converge at the optimal rates. For moderate stabilisation parameters, mass is conserved up to machine precision, as measured by the $L^2(\Omega)$-error in the divergence of the flux. Larger values of the stabilisation parameters $\xivh$ and $\xiqh$ weaken this control, whereas the \texttt{AL-BGP} method provides additional control over mass conservation.

\begin{figure}[h] \centering 
    {
        \includegraphics[height=0.052\textheight]{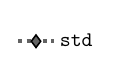}
    }\hspace{-0.7em}
    {
        \includegraphics[height=0.052\textheight]{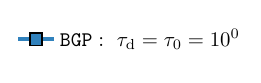}
    }\hspace{-0.7em}
    {
        \includegraphics[height=0.052\textheight]{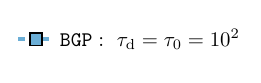}
    }\hspace{-0.7em}
    {
        \includegraphics[height=0.052\textheight]{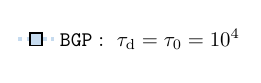}
    }\\\vspace{-1.2em}
    {
        \includegraphics[height=0.052\textheight]{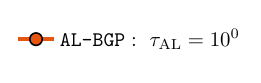}
    }\hspace{-1.0em}
    {
        \includegraphics[height=0.052\textheight]{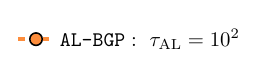}
    }\hspace{-1.0em}
    {
        \includegraphics[height=0.052\textheight]{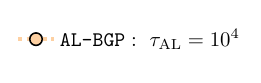}
    }\\\vspace{-1.0em}
     {
        \raisebox{0.1\textwidth}{\rotatebox[origin=t]{90}{\Small \qquad \qquad \quad $h_{\mathrm{cut}}/h = 5.0\cdot10^{-1}$}}   
     }\hspace{-0.7em} 
     {
        \includegraphics[height=0.33\textwidth]{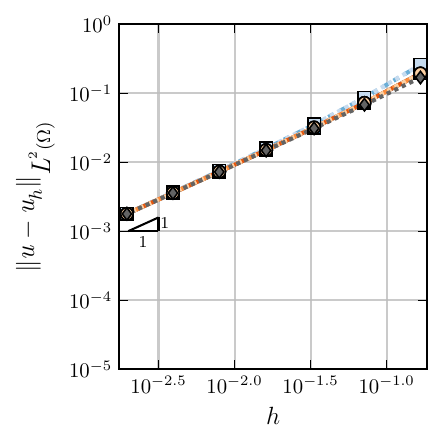}
     }\hspace{-0.9em}
     {
        \includegraphics[height=0.33\textwidth]{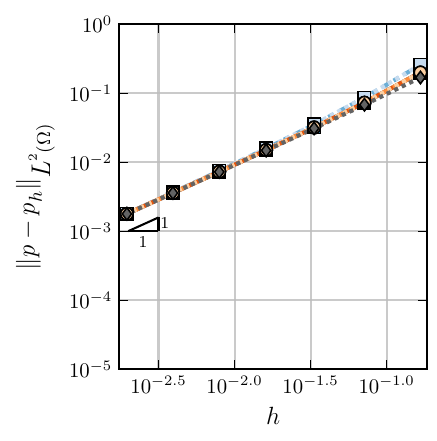}
     }\hspace{-0.9em}
     {
        \includegraphics[height=0.33\textwidth]{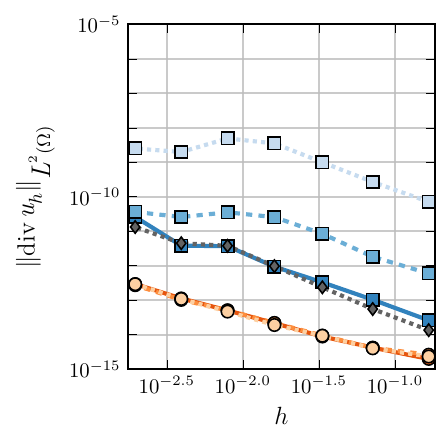}
     }
      \\\vspace{-1.2em}
    {
        \raisebox{0.1\textwidth}{\rotatebox[origin=t]{90}{\Small \qquad \qquad \quad $h_{\mathrm{cut}}/h = 5.0\cdot10^{-7}$}}   
    }\hspace{-0.7em} 
     {
        \includegraphics[height=0.33\textwidth]{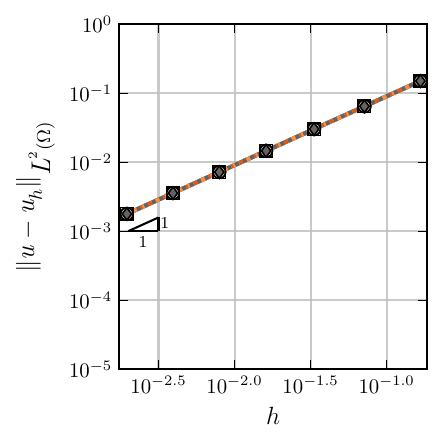}
     }\hspace{-0.9em}
     {
        \includegraphics[height=0.33\textwidth]{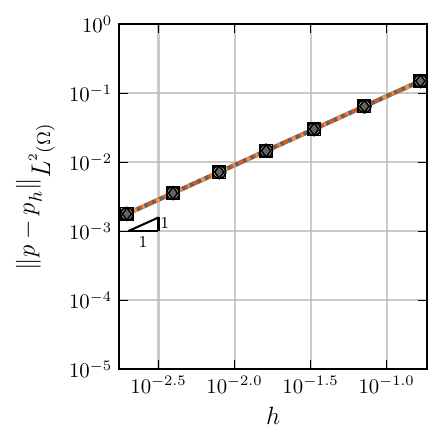}
     }\hspace{-0.9em}
     {
        \includegraphics[height=0.33\textwidth]{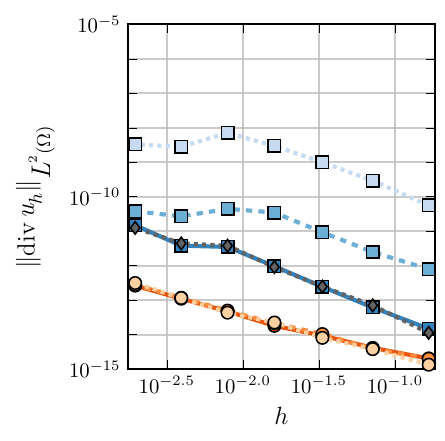}
     }\\
     \vspace{-1.2em} 
\caption{Cut square problem: $h$-convergence test using $\Gamma = \boup$. For the \texttt{AL-BGP} method, $\tau_{\mathrm{d}}= \tau_{\mathrm{0}}=10^{0}$ is used.}  \label{fig:hconvergence}
\end{figure}

\subsubsection{Sensitivity with respect to $\gamma$} We employ the manufactured solution in \eqref{eq:manufactured.sol.smooth} to investigate the influence of the penalty parameter $\gamma$ on the accuracy of the different methods for the case of pure flux boundary conditions ($\Gamma = \bouu$).
The $h$-convergence results in Figure~\ref{fig:hconvergence.B=u} show that all methods behave robustly with respect to the choice of the penalty parameter $\gamma$.
\begin{figure}[h] \centering 
    {
        \includegraphics[height=0.052\textheight]{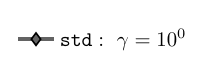}
    }\hspace{-0.7em}
    {
        \includegraphics[height=0.052\textheight]{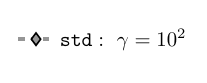}
    }\hspace{-0.7em}
    {
        \includegraphics[height=0.052\textheight]{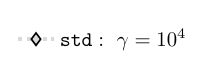}
    }\\\vspace{-1.2em}
    {
        \includegraphics[height=0.052\textheight]{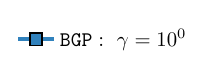}
    }\hspace{-0.7em}
    {
        \includegraphics[height=0.052\textheight]{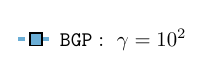}
    }\hspace{-0.7em}
    {
        \includegraphics[height=0.052\textheight]{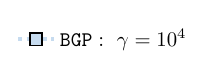}
    }\\\vspace{-1.2em}
    {
        \includegraphics[height=0.052\textheight]{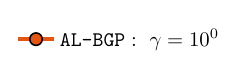}
    }\hspace{-1.0em}
    {
        \includegraphics[height=0.052\textheight]{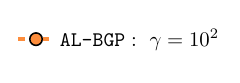}
    }\hspace{-1.0em}
    {
        \includegraphics[height=0.052\textheight]{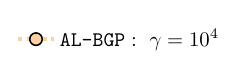}
    }\\\vspace{-1.0em}
     {
        \raisebox{0.1\textwidth}{\rotatebox[origin=t]{90}{\Small \qquad \qquad \quad $h_{\mathrm{cut}}/h = 5.0\cdot10^{-1}$}}   
     }\hspace{-0.7em} 
     {
        \includegraphics[height=0.307\textwidth]{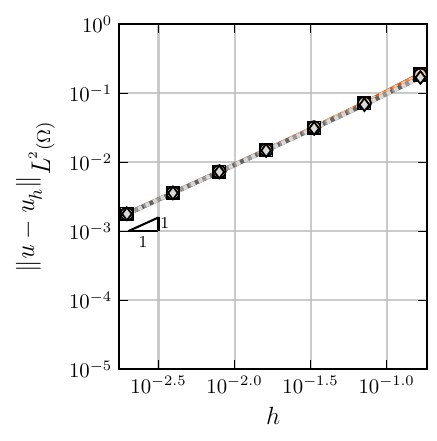}
     }\hspace{-0.9em}
     {
        \includegraphics[height=0.307\textwidth]{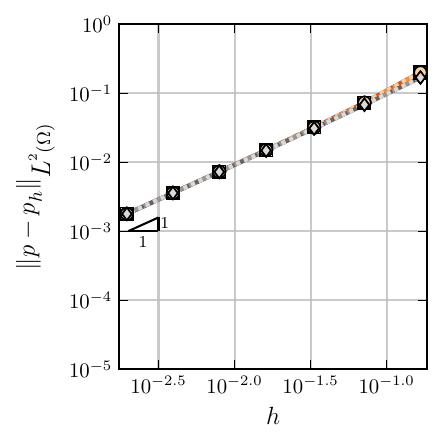}
     }\hspace{-0.9em}
     {
        \includegraphics[height=0.307\textwidth]{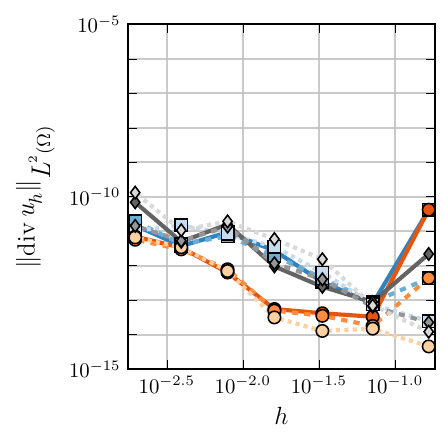}
     }
      \\\vspace{-1.2em}
    {
        \raisebox{0.1\textwidth}{\rotatebox[origin=t]{90}{\Small \qquad \qquad \quad $h_{\mathrm{cut}}/h = 5.0\cdot10^{-7}$}}   
     }\hspace{-0.7em} 
     {
        \includegraphics[height=0.307\textwidth]{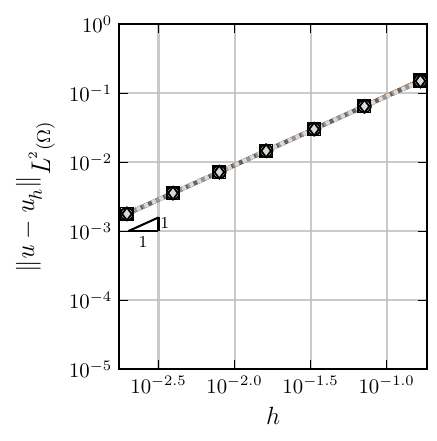}
     }\hspace{-0.9em}
     {
        \includegraphics[height=0.307\textwidth]{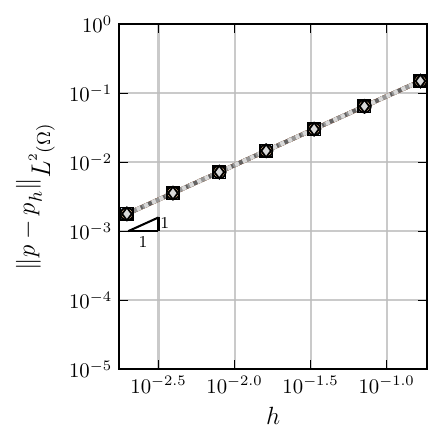}
     }\hspace{-0.9em}
     {
        \includegraphics[height=0.307\textwidth]{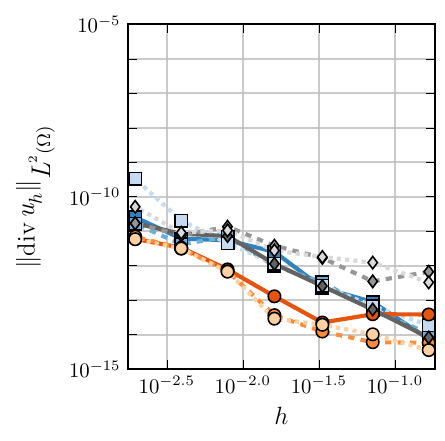}
     }\\
     \vspace{-1.2em} \centering
\caption{Cut square problem: $h$-convergence test using $\Gamma = \bouu$. The stabilisation parameters are set to $\tau_{\bullet}=10^{0}$.}
\label{fig:hconvergence.B=u}
\end{figure}
\subsubsection{Pressure robustness}
We consider the following manufactured solution to study the pressure robustness of the different methods:
\begin{equation}
    p(x,y) = \sin(\pi x)-\sin(\pi y), \qquad {u}(x,y) = (x,-y),
\end{equation}
where $p$ is not in $\Qh$, while $u$ is exactly representable in $\Vh$. Figure~\ref{fig:hconvergence.pres.robust} shows the $h$-convergence results for the case of pure pressure boundary conditions ($\Gamma = \boup$). We confirm that adding the bulk-based stabilisation does not affect pressure robustness, as the flux error norms are at machine precision. This behaviour holds for both cut configurations considered.
\begin{figure}[h] \centering 
    {
        \includegraphics[height=0.052\textheight]{fig_legend_Conv_B=p_std.pdf}
    }\hspace{-0.7em}
    {
        \includegraphics[height=0.052\textheight]{fig_legend_Conv_B=p_BGP_1.0.pdf}
    }\hspace{-0.7em}
    {
        \includegraphics[height=0.052\textheight]{fig_legend_Conv_B=p_BGP_100.0.pdf}
    }\hspace{-0.7em}
    {
        \includegraphics[height=0.052\textheight]{fig_legend_Conv_B=p_BGP_10000.0.pdf}
    }\\\vspace{-1.2em}
    {
        \includegraphics[height=0.052\textheight]{fig_legend_Conv_B=p_ALdBGP_1.0.pdf}
    }\hspace{-1.0em}
    {
        \includegraphics[height=0.052\textheight]{fig_legend_Conv_B=p_ALdBGP_100.0.pdf}
    }\hspace{-1.0em}
    {
        \includegraphics[height=0.052\textheight]{fig_legend_Conv_B=p_ALdBGP_10000.0.pdf}
    }\\\vspace{-1.0em}
     {
        \raisebox{0.1\textwidth}{\rotatebox[origin=t]{90}{\Small \qquad \qquad \quad $h_{\mathrm{cut}}/h = 5.0\cdot10^{-1}$}}   
     }\hspace{-0.73em} 
     {
        \includegraphics[height=0.307\textwidth]{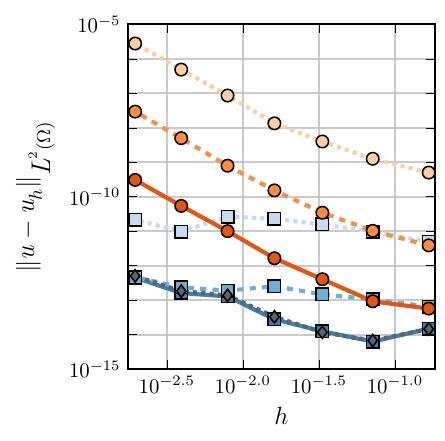}
     }\hspace{-0.92em}
     {
        \includegraphics[height=0.307\textwidth]{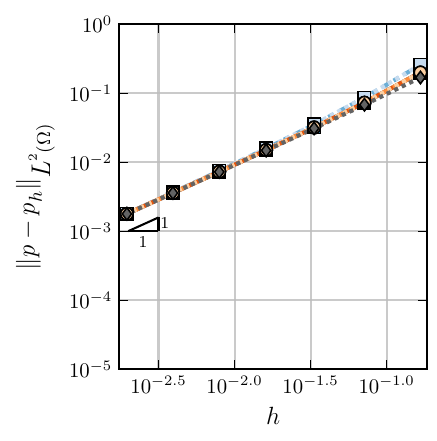}
     }\hspace{-0.92em}
     {
        \includegraphics[height=0.307\textwidth]{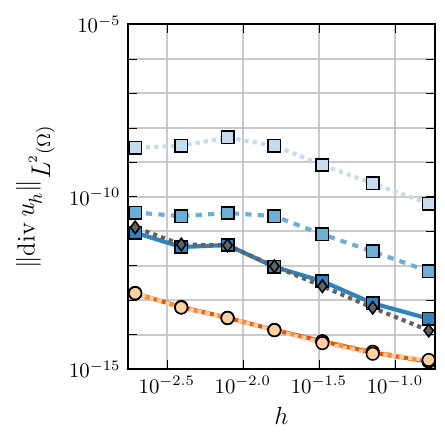}
     }
      \\\vspace{-1.2em}
    {
        \raisebox{0.1\textwidth}{\rotatebox[origin=t]{90}{\Small \qquad \qquad \quad $h_{\mathrm{cut}}/h = 5.0\cdot10^{-7}$}}   
     }\hspace{-0.73em} 
     {
        \includegraphics[height=0.307\textwidth]{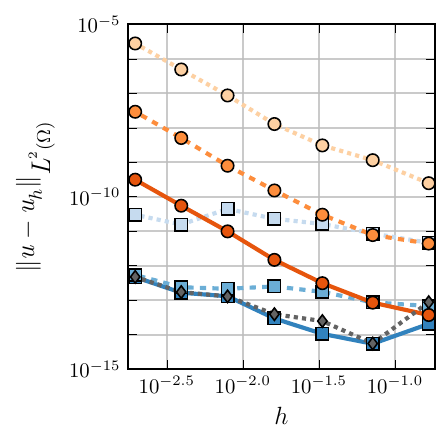}
     }\hspace{-0.92em}
     {
        \includegraphics[height=0.307\textwidth]{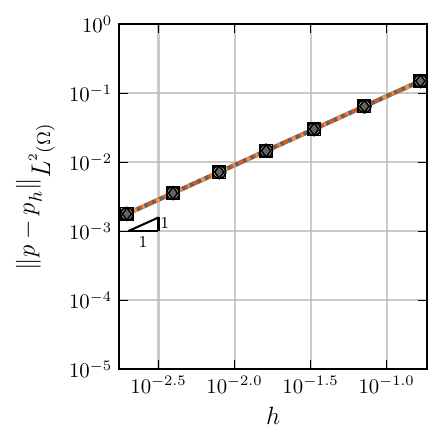}
     }\hspace{-0.92em}
     {
        \includegraphics[height=0.307\textwidth]{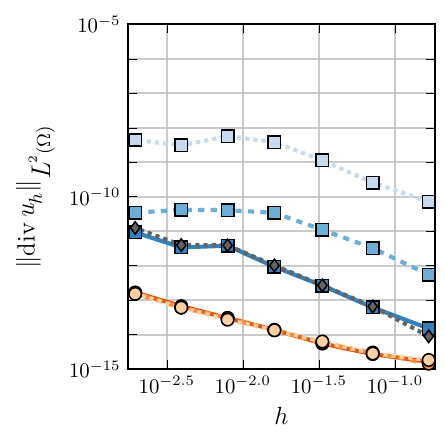}
     }\\
     \vspace{-1.2em} 
\caption{
    Cut square problem: pressure robustness $h$-convergence test using $\Gamma = \boup$. For the \texttt{AL-BGP} method, $\tau_{\mathrm{d}}= \tau_{\mathrm{0}}=10^{0}$ is used.}      \label{fig:hconvergence.pres.robust} 
\end{figure}
\subsubsection{Conditioning with respect to cut-cell length}
We again consider the manufactured solution in \eqref{eq:manufactured.sol.smooth} to study the dependence of the condition number on the cut-cell length. To this end, we fix the mesh refinement level ($n=32$) and vary $h_\mathrm{cut}$ to create physical domains with different cut-cell length ratios $h_{\mathrm{cut}}/h$.
Figure~\ref{fig:VarCut} shows that the condition number of the unstabilised method (\texttt{std}) grows inversely proportionally to the cut-cell length ratio. By contrast, both stabilised methods (\texttt{BGP} and \texttt{AL-BGP}) exhibit condition numbers that are independent of the cut-cell length ratio. All error norms are also independent of the cut-cell length ratio, confirming the theoretical results.
\begin{figure}[h] \centering 
    \begin{minipage}[]{\textwidth}
    \hspace{0.12em}
     {
      \includegraphics[height=0.307\textwidth]{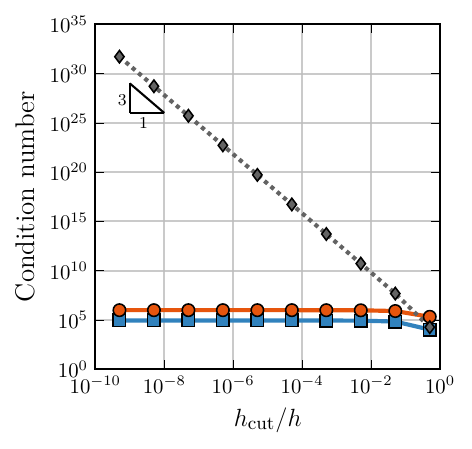}  
     }\hspace{-0.9em} 
     {
     \includegraphics[height=0.307\textwidth]{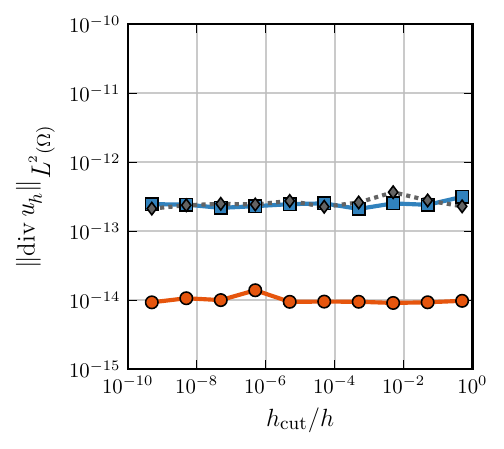}
     }\\
     {
        \includegraphics[height=0.307\textwidth]{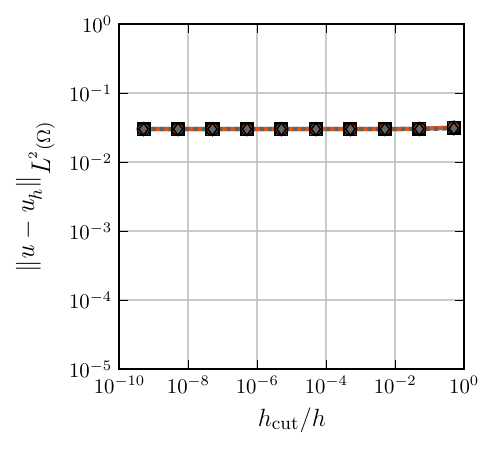}
     }\hspace{-0.9em}
     {
        \includegraphics[height=0.307\textwidth]{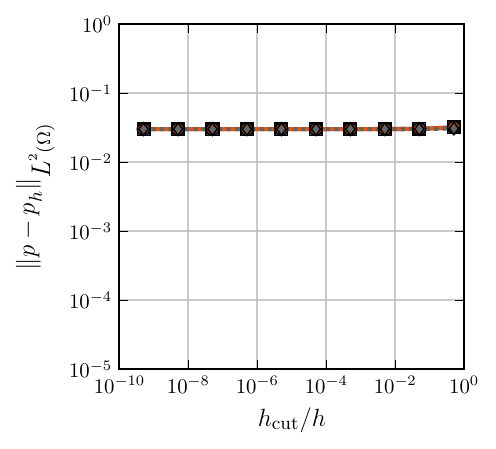}
     }\vspace{-1.0em}
    \end{minipage} \hspace{-13.0em}
    \begin{minipage}{0.12\textwidth}
        \vspace{-12em}
        \raggedright
     {
        \includegraphics[height=0.052\textheight]{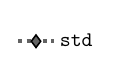}
     }\\\vspace{-1.2em}
    {
        \includegraphics[height=0.052\textheight]{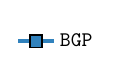}
    }\\\vspace{-1.2em}
    {
        \includegraphics[height=0.052\textheight]{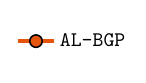}
    }
    \end{minipage}
\caption{Cut square problem for varying cut length ratios $h_{\mathrm{cut}}/h$ using $\Gamma = \boup \cup \bouu$ $(n=32)$. The stabilisation parameters are set to $\tau_{\bullet}=10^{0}$ and the penalty parameter $\gamma=10^{0}$.}
\label{fig:VarCut}
\end{figure}

\subsection{Face-based stabilisation}\label{subsec:num-results-face}
{In this section, we study the performance of the method when using the face-based stabilisation \eqref{eq:shface} described in Section~\ref{sec:face-stab}. We call this method \texttt{FGP}.} Using two numerical examples, we illustrate that this equivalent stabilisation procedure also produces optimal condition-number scalings and convergence rates in the $L^2(\Omega)$-norm, in particular for the divergence. We use triangular meshes and compare three element pairs of increasing order, namely $\mathbb{RT}_0\times \mathbb{P}_0,\, \mathbb{BDM}_1\times \mathbb{P}_0,\, \mathbb{RT}_1\times \mathbb{P}_1$.
For all examples, we choose {the penalty parameter $\gamma=1$ and all} the stabilisation parameters to be equal to one, $\tau_\bullet = 1$. We emphasise that this choice can be optimised for smaller errors. We choose the aggregation parameter $\delta = 0.25$.

These tests were carried out on a laptop with an Intel Core i7-8565U CPU and 16GB RAM, using the CutFEM library \cite{cutfemlib} written in C++. Each linear system was solved with the direct solver UMFPACK and the 1-norm estimate of the condition number was evaluated using the method \texttt{condest} in MATLAB.

\subsubsection{Flow in a rectangle generated by a linearly varying source $g$}
{For this example, we compare the unstabilised method (\texttt{std}) with the proposed \texttt{FGP} method. We numerically solve the Darcy system \eqref{eq:darcy} in $\Omega = [0,1]\times [0,0.5]$ with $\partial\Omega = \Gamma_u$}, 
where the exact solution is $u(x,y)=(x(x-1),y(y-1/2))$ and $p(x,y)=-(x^3/3-x^2/2+y^3/3-y^2/4)$.
Notice that in this example $g$ is linear. Since $\dv{\mathbb{RT}_1}\subset \mathbb{P}_1$, we expect machine-precision errors for the divergence with this choice of elements.

In Figure~\ref{fig:rect.example}, we show the optimal convergence and condition-number scaling of the \texttt{FGP} stabilisation method. Notice that the divergence converges optimally when the pressure space does not include $g$ and that machine precision is achieved (up to roundoff errors) when using $\dv{\mathbb{RT}_1}\subset \mathbb{P}_1$. 
In Table~\ref{tab:Linftydiverrors.rectangle}, we report that the same is true for {the $L^{\infty}(\Omega)$-errors of $\dv{u_h}$. The unstabilised method achieves optimal convergence rates for the lowest-order pair $\mathbb{RT}_0\times \mathbb{P}_{0}$, but the $L^{\infty}(\Omega)$-errors of $\dv{u_h}$ deteriorate for higher-order pairs due to severe ill-conditioning of the resulting linear systems.}
\begin{figure}[h] \centering 
    \begin{minipage}[]{\textwidth}
    \hspace{0.12em}
     {
      \includegraphics[height=0.307\textwidth]{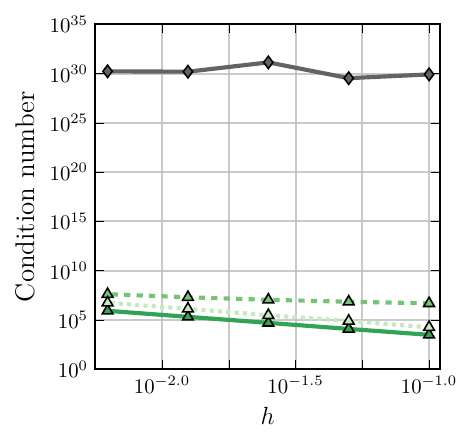}  
     }\hspace{-0.9em} 
     {
     \includegraphics[height=0.307\textwidth]{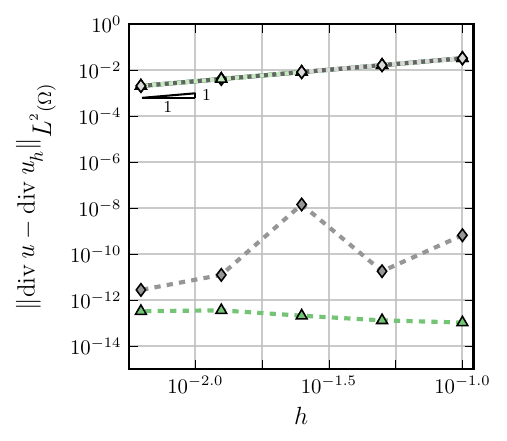}
     }\\
     {
        \includegraphics[height=0.307\textwidth]{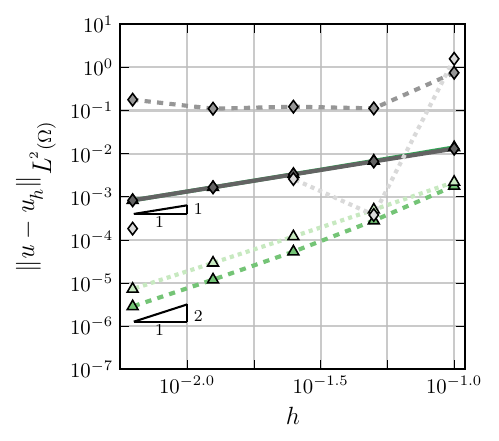}
     }\hspace{-0.9em}
     {
        \includegraphics[height=0.307\textwidth]{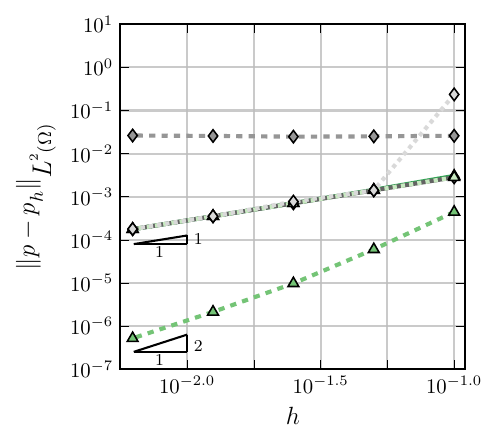}
     }\vspace{-1.0em}
    \end{minipage} \hspace{-12.5em}
    \begin{minipage}{0.12\textwidth}
        \vspace{-9em}
        \raggedright
     {
        \includegraphics[height=0.052\textheight]{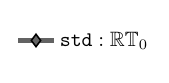}
     }\\\vspace{-1.2em}
    {
        \includegraphics[height=0.052\textheight]{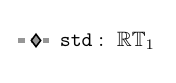}
    }\\\vspace{-1.2em}
    {
        \includegraphics[height=0.052\textheight]{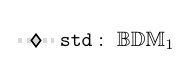}
    }\\\vspace{-1.2em}
    {
        \includegraphics[height=0.052\textheight]{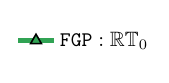}
     }\\\vspace{-1.2em}
    {
        \includegraphics[height=0.052\textheight]{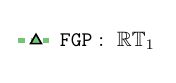}
    }\\\vspace{-1.2em}
    {
        \includegraphics[height=0.052\textheight]{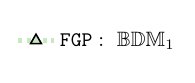}
    }
    \end{minipage}
\caption{Rectangle problem. Unstabilised method (\texttt{std}) compared with the \texttt{FGP} stabilisation method. For the \texttt{FGP} method, $\tau_{\bullet}= 1$ and $\delta = 0.25$ are used.}
\label{fig:rect.example}
\end{figure}
\begin{table}[h]
\caption{{Rectangle problem. $L^{\infty}(\Omega)$-errors for  $\dv{u_h}$.}}
\label{tab:Linftydiverrors.rectangle}
\begin{tabular}{@{}lllllll@{}}
\toprule
$h$ & \texttt{FGP} $(\mathbb{RT}_0)$ & \texttt{std} $(\mathbb{RT}_0)$ & \texttt{FGP} $(\mathbb{BDM}_1)$ & \texttt{std} $(\mathbb{BDM}_1)$ & \texttt{FGP} $(\mathbb{RT}_1)$ & \texttt{std} $(\mathbb{RT}_1)$ \\
\midrule
0.1     & 0.1333 & 0.1333 & 0.1333 & 128.1    & 1.7625e-12 &  25395      \\
0.05    & 0.0667 & 0.0667 & 0.0667 & 1.5      & 3.7923e-12 &  3585.4      \\
0.025   & 0.0333 & 0.0333 & 0.0333 & 96.35    & 1.3436e-11 &  5.2429e+06 \\
0.0125  & 0.0167 & 0.0167 & 0.0167 & 0.1341   & 2.5998e-11 &  4352.5     \\
0.00625 & 0.0083 & 0.0083 & 0.0083 & 100.16   & 3.9275e-11 &  319.51     \\
\bottomrule
\end{tabular}
\end{table}

\subsubsection{{Transmission over an interface}}
{In this experiment, we consider the Darcy interface problem in \cite[Equations (2.10)--(2.15)]{Frachon2024-fg}, originally introduced in \cite{DAngelo-EsaimMathModelNumerAnal-2012r}. We combine the unfitted interface discretisation from~\cite{Frachon2024-fg} with our approach for the weak imposition of boundary conditions, thereby obtaining a fully unfitted scheme.} {Let the domain $\Omega$ be an annular domain centred at $x_c = (0.5, 0.5)$, with inner radius $r_1 = 0.15$ and outer radius $r_2 = 0.45$. This domain is embedded in the background unit square $\Omega_h^{\mathrm{bg}} = [0, 1]^2$.} 
We introduce a circular interface $\Gamma^{\mathrm{int}}$ located at the midpoint radius $r_{\mathrm{int}} = 0.3$. This interface partitions the domain $\Omega$ into two subdomains: $\Omega_1$, the inner region where $r_1 < r < r_{\mathrm{int}}$, and $\Omega_2$, the outer region where $r_{\mathrm{int}} < r < r_2$. 
{We consider mixed boundary conditions, i.e. $\partial\Omega = \Gamma_u \cup \Gamma_p$, where $\Gamma_p$ and $\Gamma_u$ are defined at the inner and outer boundaries of the annular domain $\Omega$, respectively.} 
We impose {pressure-based} transmission conditions across the interface $\Gamma^{\mathrm{int}}$. Let the average operator be defined as $\{p\} = (p|_{\Omega_1} + p|_{\Omega_2})/2$. The jump and average conditions are governed by:
\[
    [p] = \eta_{\mathrm{int}} \{u \cdot n\}, \qquad
    \{p\} = \hat{p} + \frac{\eta_{\mathrm{int}}}{8} [u \cdot n].
\]
The interface-dependent parameters $\eta_{\mathrm{int}}$ and $\hat{p}$ are defined as:
\[
    \eta_{\mathrm{int}} = \frac{2 r_{\mathrm{int}}}{4 \cos(r_{\mathrm{int}}^2) + 3}, \qquad
    \hat{p} = \frac{19r_{\mathrm{int}}^2 + 12\sin(r_{\mathrm{int}}^2) + 8\sin(2r_{\mathrm{int}}^2) + 24r_{\mathrm{int}}^2\cos(r_{\mathrm{int}}^2)}{4r_{\mathrm{int}}^2(4\cos(r_{\mathrm{int}}^2)+3)}.
\]
We consider the following discrete formulation of the problem:
Find $(u_h,p_h)\in \Vh \times \Qh$ such that
\begingroup
\mathtoolsset{showonlyrefs=false}
\begin{subequations}\label{eqs:discrete-interface}
\begin{align}
	a_h(u_h,v_h) + \tilde \bbf_h(v_h,p_h) &= (f,v_h)_{\Omega_1 \cup \Omega_2} 
	- (p_\Gamma,v_h\cdot n)_{\Gamma_p} + \gamma h^{-1} (u_\Gamma,v_h \cdot n )_{\bouu} \nonumber \\
    &\quad - (\hat{p}, [v_h\cdot n])_{\Gamma^{\mathrm{int}}}, \quad \forall v_h\in \Vh, \\
	b_h(u_h,q_h) &= (g,q_h)_{\Omega_1 \cup \Omega_2}, \quad \forall q_h\in \Qh,
\end{align}
\end{subequations}
\endgroup
where  
\begin{align*}
	a_h(u_h,v_h) &:= (\eta u_h,v_h)_{\Omega_1 \cup \Omega_2} + \xivh s_h^d(u_h,v_h) + \gamma h^{-1} (u_h \cdot n, v_h \cdot n)_{\bouu}\\
    &\quad + (\eta_{\mathrm{int}} \{u_h\cdot n\},\{v_h\cdot n\})_{\Gamma^{\mathrm{int}}} + (\eta_{\mathrm{int}} [u_h\cdot n], [v_h\cdot n])_{\Gamma^{\mathrm{int}}}.
\end{align*}
{Equations \eqref{eqs:discrete-interface} define {the \texttt{FGP} method}. To obtain the unstabilised method (\texttt{std}), one sets all stabilisation parameters $\tau_\bullet = 0$.}

To validate the method, we utilise the following manufactured exact solutions. The pressure field $p$ is given by:
\begin{equation*}
    p(x,y) = \sin(x^2 + y^2) + 
    \begin{cases} 
        \frac{r^2}{2r_{\mathrm{int}}^2} + \frac{3}{2} & \text{in } \Omega_1, \\
        \frac{r^2}{r_{\mathrm{int}}^2} & \text{in } \Omega_2.
    \end{cases}
\end{equation*}
The flux field $u$ is defined as:
\begin{equation*}
    u(x,y) = -\frac{(x, y)}{r_{\mathrm{int}}^2}
    \begin{cases} 
        1 + 2\cos(x^2 + y^2) & \text{in } \Omega_1, \\
        2(1 + \cos(x^2 + y^2)) & \text{in } \Omega_2.
    \end{cases}
\end{equation*}
Here, note that $\dv{u}$ is quadratic, so we investigate its convergence rate. 
In Figure~\ref{fig:interface.example}, we show that we obtain optimal convergence rates and condition-number scaling for the considered element pairs using {the \texttt{FGP} method}.
{As in the previous example, we see in Table~\ref{tab:Linftydiverrors.interface} that the divergence also converges optimally in $L^{\infty}(\Omega)$ for {the \texttt{FGP} method}.}
\begin{figure}[h] \centering 
    \begin{minipage}[]{\textwidth}
    \hspace{0.12em}
     {
      \includegraphics[height=0.307\textwidth]{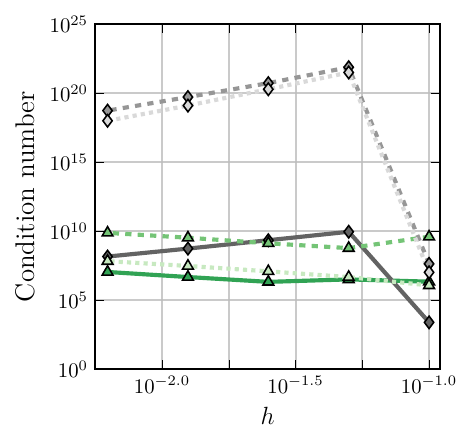}  
     }\hspace{-0.9em} 
     {
     \includegraphics[height=0.307\textwidth]{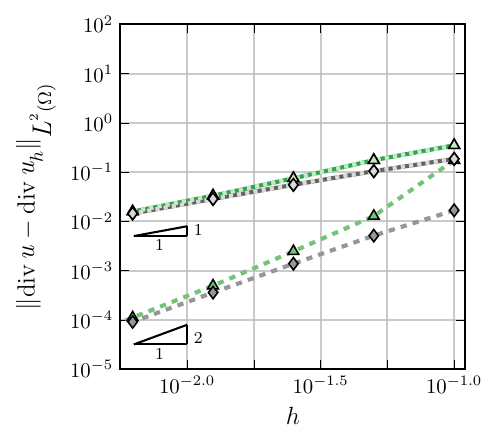}
     }\\
     {
        \includegraphics[height=0.307\textwidth]{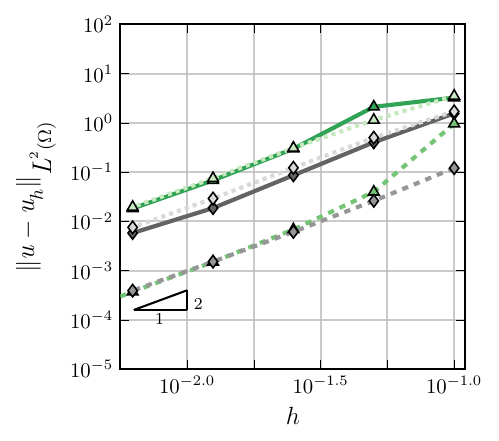}
     }\hspace{-0.9em}
     {
        \includegraphics[height=0.307\textwidth]{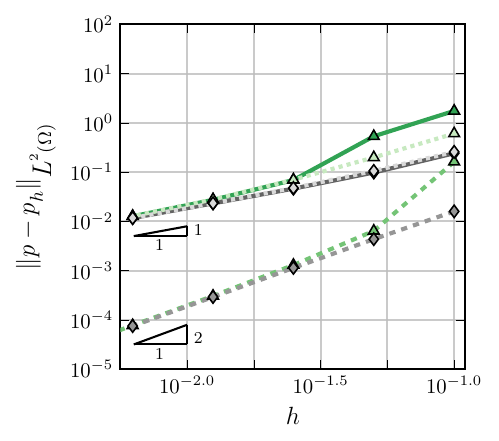}
     }\vspace{-1.0em}
    \end{minipage} \hspace{-12.5em}
    \begin{minipage}{0.12\textwidth}
        \vspace{-9em}
        \raggedright
     {
        \includegraphics[height=0.052\textheight]{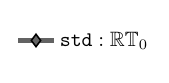}
     }\\\vspace{-1.2em}
    {
        \includegraphics[height=0.052\textheight]{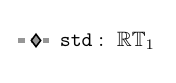}
    }\\\vspace{-1.2em}
    {
        \includegraphics[height=0.052\textheight]{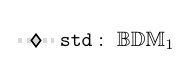}
    }\\\vspace{-1.2em}
    {
        \includegraphics[height=0.052\textheight]{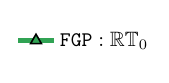}
     }\\\vspace{-1.2em}
    {
        \includegraphics[height=0.052\textheight]{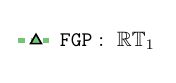}
    }\\\vspace{-1.2em}
    {
        \includegraphics[height=0.052\textheight]{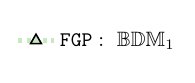}
    }
    \end{minipage}
\caption{{Interface problem. Comparison between \texttt{std} and \texttt{FGP}. For the stabilised method, $\tau_{\bullet}= 1$ and $\delta = 0.25$ is used.}}
\label{fig:interface.example}
\end{figure}
\begin{table}[h]
\caption{{Interface problem. $L^{\infty}(\Omega)$-errors for $\dv{u_h}$.}}
\label{tab:Linftydiverrors.interface}
\centering
\begin{tabular}{@{}llllll@{}}
\toprule {$h$}
& $1/10$ & $1/20$ & $1/40$ & $1/80$ & $1/160$ \\
\midrule
\texttt{FGP} $(\mathbb{RT}_0)$ & 2.4550 & 2.0023 & 1.1309 & 0.6214 & 0.3288 \\
\texttt{FGP} $(\mathbb{BDM}_1)$ & 2.4550 & 2.0022 & 1.1309 & 0.6214 & 0.3288 \\
\texttt{FGP} $(\mathbb{RT}_1)$ & 2.2330 & 0.2669 & 0.0770 & 0.0203 & 0.0054 \\
\bottomrule
\end{tabular}
\end{table}
\section{Conclusions}
\label{sec:conclusions}

We have introduced divergence-free unfitted mixed finite element discretisations for Darcy flow that preserve pointwise discrete mass conservation and remain stable for arbitrarily small cut cells. The formulation admits both cell-wise (bulk) and face-based ghost-penalty realisations. We proved stability and error estimates, including pressure-robust flux bounds for pure pressure boundary conditions. The error estimates are optimal in all considered cases, except in the presence of flux boundary conditions combined with trimmed polynomial spaces. The numerical experiments show that the methods deliver optimal convergence, cut-independent conditioning, and robust performance across different boundary conditions and parameter regimes.

\section{Acknowledgements}
This research was partially funded by the Australian Government through the Australian Research Council (project numbers DP210103092 (SB), DP220103160 (SB, AB \& RRB), and FT220100496 (RRB)). EN and SZ were supported by the Swedish Research Council (Grant No. 2022–04808) and the Wallenberg Academy Fellowship (KAW 2019.0190). Computational resources were provided by the Australian Government through the National Computational Infrastructure (NCI) under the National Computational Merit Allocation Scheme (NCMAS) and the ANU Merit Allocation Scheme, and by Monash eResearch through the Monash NCI scheme for HPC services. The authors also wish to thank MATRIX (Mathematical Research Institute) for their generous support in hosting the workshop "Numerical Analysis of Interface and Multiphysics Problems" (5–16 May 2025) in Creswick, Victoria, where the discussions and collaborations that led to this work were initiated.
\printbibliography

\bigskip\noindent\textbf{Reproducibility.} 
The routines used to generate the numerical results in Section \ref{subsec:num-results-bulk} are available at \url{https://doi.org/10.5281/zenodo.19199178}. The simulations in Section \ref{subsec:num-results-face} used a different computational framework.

\end{document}